\documentclass{article}
\usepackage[numbers]{natbib}
\usepackage[T2A]{fontenc}
\usepackage{shadethm}
\usepackage{amsthm}
\usepackage{float}
\usepackage{bm}
\usepackage{mathtools}
\usepackage{graphicx}
\usepackage{subfig}

\usepackage{mdframed}
\usepackage{lipsum}

\usepackage{stmaryrd}
\usepackage{mathrsfs} 
\usepackage{amsmath}
\usepackage{amssymb}
\usepackage{wasysym}
\usepackage[hyperfootnotes=false]{hyperref}
\usepackage{geometry}
\usepackage{chngcntr}
\counterwithin*{equation}{section}

\usepackage{enumitem}

\newcommand{\R}{\mathbb{R}}

\usepackage{cleveref}
\DeclareMathAlphabet{\mathpzc}{OT1}{pzc}{m}{it}

\renewcommand{\theequation}{\arabic{section}.\arabic{equation}}

\theoremstyle{plain}
\newtheorem{thm}{Theorem}[section] 
\newtheorem{lem}[thm]{Lemma}

\theoremstyle{definition}

\theoremstyle{remark}
\newtheorem{rem}[thm]{Remark}

\usepackage[utf8]{inputenc}
\usepackage[german,english,russian]{babel}

\title{Inverse problems for quasi-linear elliptic systems\\ modeling electrolysers}
\author{Giovanni S. Alberti\footnote{\textit{E-mail address:} \href{mailto:giovanni.alberti@unige.it}{giovanni.alberti@unige.it}}, Wadim Gerner\footnote{\textit{E-mail address:} \href{mailto:wadim.gerner@edu.unige.it}{wadim.gerner@edu.unige.it}}, Matteo Santacesaria\footnote{\textit{E-mail address:} \href{mailto:matteo.santacesaria@unige.it}{matteo.santacesaria@unige.it}}}
\date{\begin{center}
		{\footnotesize	MaLGa Center, Department of Mathematics, Department of Excellence 2023-2027, \\University of Genoa, Via Dodecaneso 35, 16146 Genova, Italy}
	\end{center}}
\begin{document}\selectlanguage{english}
\maketitle

\begin{abstract}
We investigate the electrochemical processes within an electrolyser cell, which are modelled by a coupled system of second-order quasi-linear elliptic PDEs. In this context, we study an inverse problem aiming to reconstruct both the non-linear diffusion coefficients and the phenomenological relation defining the electric potential. Our main results state that boundary measurements alone are not enough to reconstruct these non-linear quantities. However, we show that a combination of boundary and interior measurements allow for their unique reconstruction. To achieve this result we generalise a linearisation result in the context of the scalar quasi-linear Calder\'{o}n problem, [Sun, Math. Z. 221 (1996)], to the setting of a system of PDEs with non-local nonlinearities. In contrast to the Calder\'{o}n case, the generalised linearisation does not ``freeze" the coefficients. We show that interior measurements are precisely what is required to achieve this freezing and thus enable the unique reconstruction.
        \end{abstract}


		\section{Introduction}
		
		Hydrogen is utilised in many industrial processes \cite{HauLohWat11,QaGuAb24,Lewis21,KumKum15,Block04}, including ammonia production, which is used in fertilisers \cite{Smil01}, and as such sustains the growing global need for food. Many common ways to produce hydrogen involve carbon dioxide as a by-product \cite{Wis19}, which leads to environmental pollution. Promising alternatives are proton exchange membrane (PEM) electrolysis and anion exchange membrane (AEM) electrolysis. Neither produces any carbon dioxide or other harmful by-products, provided renewable energy is used to generate the electricity required for the reactions. The hydrogen produced in this way is also known as green hydrogen \cite{SquMagNic23}. PEM water electrolyser technology goes back to the 1960s \cite{Grubb59A,Grubb59B,Grubb60,RusNutFic73} and is well-developed. Its main disadvantage is that its membrane contains materials such as platinum, iridium or ruthenium as catalysts, which makes these devices expensive \cite{KimLeeQiKim24}. In contrast, AEM electrolysis requires catalysts that are less costly and may therefore provide a good alternative to PEM electrolysers \cite{KimLeeQiKim24}. However, research on AEM electrolysers began more recently \cite{WuScott11,KumLim22,Leng12} and as such has not yet reached the same maturity as its PEM counterpart. The underlying working principles pertaining to PEM electrolysers and AEM electrolysers are the same. Electricity is injected into a device containing water, which in turn is eventually split into hydrogen and oxygen by means of electrochemical reactions. The hydrogen can then be extracted for industrial applications.
		
		The goal of the present work is to study a simplified PEM\slash AEM-electrolyser model and to understand what types of measurements must be performed in order to uniquely determine the relationship between the electric potential $\phi$ within the electrolyser and the temperature and the distinct ion concentrations within the device. This has the potential to help improve AEM electrolyser designs in order to increase their lifetime.

        Our main result states that a combination of boundary and interior measurements are enough to uniquely reconstruct the unknown coefficient functions in a  quasi-linear elliptic system of PDEs modeling the electrolyser. We further show that a unique reconstruction cannot be achieved by boundary measurements alone. The main ingredient in the unique reconstruction result is a generalisation of a linearisation lemma by Sun \cite{Sun96}. The original setup \cite{Sun96} dealt with the quasi-linear version of the scalar Calderón problem. Our generalisation, cf.\ \Cref{S3L8}, applies to systems of equations and also allows for the presence of non-local non-linearities which is crucial for the system of equations studied in the present work. This linearisation lemma may be useful in the analysis of other inverse problems and is therefore of independent interest. In contrast to the scalar Calderón problem \cite{Sun96}, the linearisation procedure for our specific system of equations will not freeze the variables. Instead, due to the non-local dependence, the linearised variables remain space dependent. We then make use of the interior measurements in order to achieve a freezing of the variables.

        The structure of the paper is as follows: In \Cref{Section2} we describe the system of equations and the measurements that we consider throughout the manuscript. In \Cref{Section3} we present the main results and discuss their relation to previous literature. In \Cref{Section4} we present the proofs of the main results, including our linearisation lemma. The appendix discusses the derivation of the model studied in this paper in more detail, and potential extensions.

        \section{Setup of the problem}
        \label{Section2}
		\subsection{The equations and the unknowns}
		\label{Inverse}
		We use the following equations to model our electrolyser device:
		\begin{gather}
			\nonumber
			\partial_tc_i-\operatorname{div}(D_i\nabla c_i)=g_i\text{ for }1\leq i\leq M\text{, } \operatorname{div}(\epsilon \nabla \phi)=q\cdot c.
		\end{gather}
        They arise from the macroscopic Maxwell equations, which describe the electric potential, and Fick's laws, which describe the evolution of the ion concentrations. These equations are coupled via the free charges and currents that are generated by the charged ions. See \Cref{Model} for a detailed derivation of these equations, as well as \Cref{Electrostatic} for a possible extension  of our model.

		The symbols in the equations have the following meanings:
		\begin{itemize}[noitemsep]
			\item $M$ is the number of distinct particle species in our electrolyser;
			\item $c_i=c_i(t,x)$ denotes the concentration of the $i$-th particle species and we set $c=(c_1,\dots,c_M)$;
			\item $T=T(t,x)$ is the (scalar) temperature of the system;
			\item $D_i=D_i(t,c,T,x)$ is the scalar diffusion coefficient of the $i$-th species, which we assume to be bounded below and above by some positive constants;
			\item $g_i=g_i(t,c,T,x)$ is the source-term corresponding to the $i$-th species production\slash annihilation;
			\item $\epsilon=\epsilon(x)$ is the electric permeability, which is assumed to be bounded below and above by some positive constant;
            \item $\phi=\phi(t,c,T,x)$ is the electric potential;
			\item $q_i\in \mathbb{R}$ is the electric charge of the $i$-species and $q\cdot c$ is the standard Euclidean inner product of $q=(q_1,\dots,q_M)$ and $c=(c_1,\dots,c_M)$.
		\end{itemize}

		 A priori, we do not have knowledge of the diffusion coefficients $D_i$, of the precise structure of the sources $g_i$, or of the potential $\phi$. As indicated in (\ref{S1E15}), some additional structural assumptions on the electric potential can be made in principle. In turn, using Nernst's equation, Butler-Volmer's equation and Tafel's equation, some more specific expressions for the distinct electric potential contributions may be derived \cite{Lawand24,MajHaasEllNaz23,SeiMitBon25}. However, in any case, measurements are required to determine the precise values of the parameters appearing in these models. In the present work, we do not make any structural assumptions on the electric potential $\phi$; we only assume that it can be expressed as a function of time, position, temperature and the particle-concentrations.
		
		In our application, we know the number of distinct particle species $M$ and their corresponding charges $q=(q_1,\dots,q_M)$ (in Coulomb). We assume that $\epsilon=\epsilon(x)$ is a known quantity. In a zero order approximation, one could assume that $\epsilon$ is constant throughout the domain $\Omega$ occupied by the electrolyser, or that it is a simple function with constant (but possibly distinct) values in certain distinguished areas of the domain, for instance the membrane, the anode, the cathode and the area containing the electrolyte.
		
		The inverse problem we want to deal with is to reconstruct $D_i$, $g_i$ and $\phi$ as functions of $t,c,T,x$ from certain measurements, see the coming \Cref{Measure} for a discussion about the type of measurements we can perform. We notice that we have $M+1$ equations in total and $2M+1$ unknown functions ($M$ functions $D_i$, $M$ functions $g_i$ and the potential $\phi$). It is therefore not clear to what extent all these quantities may be reconstructed; possibly, some additional structural assumptions on the sources may have to be made in order to be able to recover all these quantities uniquely. The main reconstruction result of the present work focuses on the special case $g_i=0$ for all $i$, see \Cref{Static}, to gain first insights into this inverse problem, while the case of non-zero sources is left for future investigations.
        
		\subsection{The measurements}
		\label{Measure}
		In order to obtain the boundary values of $\phi$ we can, as discussed in detail in \Cref{Potential}, use a voltmeter to measure the voltage between any two points on the boundary. Current density-voltage profiles are also used in AEM electrolyser cells to evaluate their performance \cite{Lawand24}.
		
		We then need to specify the boundary conditions that we may prescribe and the boundary measurements that we may perform on $T$ and $c$.
		Let us start with the temperature. We are able to prescribe $T|_{\partial\Omega}$ by controlling the temperature of the electrolyte at the boundary. In addition, for any boundary point $x\in \partial\Omega$, we can measure the temperature at nearby points that are displaced in the normal direction $\mathcal{N}(x)$, i.e.\ at $x+r\mathcal{N}(x)$ for $0<r\ll 1$, where $r$ is the distance of this point from $\partial\Omega$. For small $r$ we can then take
		\begin{gather}
			\nonumber
			\frac{T(x+r\mathcal{N}(x))-T(x)}{r}\approx \mathcal{N}(x)\cdot \nabla T(x)
		\end{gather}
		and hence measure the normal trace of the temperature on $\partial\Omega$.

        Regarding the concentrations, we may consider here the simplest electrolyser setup, cf.\ \cite{Leng12}, which contains $\mathrm{H_2 O}$, $\mathrm{OH}^{-}$, $\operatorname{O}_2$, $\operatorname{H}_2$ and $\operatorname{e}^{-}$:
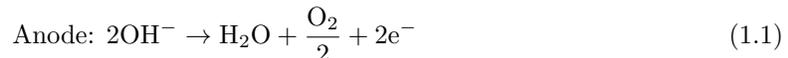
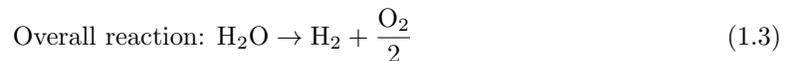
        \begin{gather}
			\label{S1E1}
			\text{Anode: }2\text{OH}^{-}\rightarrow \text{H}_2\text{O}+\frac{\text{O}_2}{2}+2\text{e}^-
			\\
			\label{S1E2}
			\text{Cathode: }2\text{H}_2\text{O}+2\text{e}^-\rightarrow \text{H}_2+2\text{OH}^{-}
			\\
			\label{S1E3}
			\text{Overall reaction: }\text{H}_2\text{O}\rightarrow \text{H}_2+\frac{\text{O}_2}{2}
		\end{gather}
        Regarding the boundary conditions that we may control let us first consider $\mathrm{H_2 O}$. We can control the water inflow (at least on a part of $\partial\Omega$), which amounts to prescribing $\mathcal{N}\cdot (D_{\mathrm{H_2 O}}\nabla c_{\mathrm{H_2 O}})$ (notice that in our simplified model we assume that the liquid water also follows Fick's law). In addition, in an ideal situation, all the $\operatorname{OH}^{-}$ that is produced at the cathode recombines at the anode, cf.\ (\ref{S1E1}) and (\ref{S1E2}), and so no $\operatorname{OH}^{-}$ escapes the system, i.e.\ $\mathcal{N}\cdot \left(D_{\operatorname{OH}^{-}}\nabla c_{\operatorname{OH}^{-}}\right)=0$. Concerning $\operatorname{O}_2$ and $\operatorname{H}_2$, we observe that they are removed from the system through a pipe and so, by adjusting the strength at which they are removed from the system, we may also prescribe (at a boundary portion) the values of $\mathcal{N}\cdot \left(D_{\operatorname{H}_2}\nabla c_{\operatorname{H_2}}\right)$ and $\mathcal{N}\cdot \left(D_{\operatorname{O}_2}\nabla c_{\operatorname{O_2}}\right)$. Finally, regarding $\operatorname{e}^{-}$, the induced current that we can control corresponds to $\mathcal{N}\cdot\left(D_{\operatorname{e}^{-}}\nabla c_{\operatorname{e}^{-}}\right)$ (we assume that the electrons also follow Fick's law).
		
		We conclude that in each case $\mathcal{N}\cdot (D_i\nabla c_i)$ can be prescribed (at a boundary portion) or is fixed and known a priori. We assume for simplicity that we have the freedom to prescribe $\mathcal{N}\cdot (D_i\nabla c_i)$ on all of $\partial\Omega$ and for all $1\leq i\leq M$.
		
		Further, near the entrance of the pipe\slash exit of the electrolyser at which the hydrogen $\operatorname{H}_2$ and oxygen $\operatorname{O}_2$ leave the electrolyser cell, we may measure the ion-concentrations so that we may measure the responses $c_i|_P$ for some boundary portion $P\subset \partial\Omega$. To simplify the situation we assume throughout this work that $P=\partial\Omega$.

        As we shall see, boundary measurements alone are not enough to determine the functional dependence of the potential $\phi$ in the interior of the cell, cf.\ \Cref{S2T1}. Therefore, we will also consider a hybrid inverse problem where we assume that it is possible for us to measure the temperature at interior points of the domain $\Omega\subseteq\mathbb R^3$ occupied by the electrolyser cell. We point out that it is indeed possible to measure the temperature in the interior of electrolyser cells in practice \cite{Hua24}.
		
		To summarise, we may prescribe and measure the following quantities
		\begin{align}
			\label{S1E16}
			&\text{Prescribe: }&&\left(T,\mathcal{N}\cdot (D_1\nabla c_1),\dots,\mathcal{N}\cdot (D_M\nabla c_M)\right)\text{ on }[0,\infty)\times \partial\Omega,
			\\
			\label{S1E17}
			&\text{Measurements I: }&&\left(\mathcal{N}\cdot \nabla T,c_1,\dots,c_M\right)\text{ on }[0,\infty)\times\partial\Omega,
			\\
			\label{S1E18}
			&\text{Measurements II: }&&\phi(t,c(t,x),T(t,x),x)-\phi(t,c(t,y),T(t,y),y)\text{ for all }x,y\in \partial\Omega,t\geq 0,
			\\
			\label{S1E19}
			&\text{Measurements III: }&&T(t,x)\text{ for all }t\geq 0\text{, }x\in \Omega.
		\end{align}
		In practice, one has only a finite sample of these measurements, but we assume here that we have access to infinitely many measurements, as is typical in the analysis of inverse boundary value problems.
		
		 In the setting of the time-dependent problem one needs to additionally prescribe the initial condition at time $t=0$ to obtain a well-posed system of equations. However, in the present manuscript we only deal with the static problem, so we omit this discussion.
		\subsection{The static source-free case}
		\label{Static}
		In the present work we want to initiate the analysis of the inverse problem described in \Cref{Inverse} assuming that we can prescribe and measure the quantities of interest as described in (\ref{S1E16})-(\ref{S1E19}). We want to assume here that our electrolyser cell reached an equilibrium state and therefore all quantities involved are time-independent and satisfy the corresponding static equations
		\begin{gather}
			\label{S1E20}
			-\operatorname{div}(D_i(c(x),T(x),x)\nabla c_i(x))=g_i(c(x),T(x),x)\text{ and }\operatorname{div}(\epsilon(x) \nabla (\phi(c(x),T(x),x))=q\cdot c(x).
		\end{gather}
		Further, we make at some points the (unphysical) assumption that $g_i=0$ for all $i$. The situation in which $g_i\neq 0$ is much more complex than its source-free counterpart.
		We hence focus at some instances on the following equations
		\begin{gather}
			\label{S1E21}
			\operatorname{div}(D_i(c(x),T(x),x)\nabla c_i(x))=0\text{, }1\leq i\leq M\text{ and }\operatorname{div}(\epsilon(x) \nabla (\phi(c(x),T(x),x))=q\cdot c(x),
		\end{gather}
		where $\epsilon=\epsilon(x)$ is assumed to be known, $D_i=D_i(c,T,x)$ and $\phi=\phi(c,T,x)$ are unknown functions of the concentrations, the temperature, and the position. Further, $q\in \mathbb{R}^M$ is a known constant vector.
        \bigskip

        \noindent\fbox{\begin{minipage}{0.985\textwidth}
\paragraph{The inverse problem.} We consider the reconstruction of the functions $D_i$ and $\phi$ according to:
        \begin{align}
			\label{S1E22}
			&\text{Prescribe: }&&\left(T|_{\partial\Omega},c_1|_{\partial\Omega},\dots,c_M|_{\partial\Omega}\right),
			\\
			\label{S1E23}
			&\text{Measurements I: }&&\left(\mathcal{N}\cdot \nabla T,\mathcal{N}\cdot (D_1\nabla c_1),\dots,\mathcal{N}\cdot (D_M\nabla c_M)\right),
			\\
			\label{S1E24}
			&\text{Measurements II: }&&\phi(c(x),T(x),x)-\phi(c(y),T(y),y)\text{ for all }x,y\in \partial\Omega,
			\\
			\label{S1E25}
			&\text{Measurements III: }&&T(x)\text{ for }x\in\Omega.
		\end{align}
\end{minipage}}

        \medskip\noindent
		Note that, compared to (\ref{S1E16})-(\ref{S1E19}), we prescribe here the Dirichlet data and measure the Neumann data. This is due to the fact that from a mathematical perspective it is easier to deal with the quasi-linear forward problem for Dirichlet boundary conditions rather than Neumann boundary conditions. This is common practice in the mathematical analysis of the Calderón problem \cite{calderon1980inverse,sylvester1987global}, an inverse problem with boundary data modeled by a scalar elliptic PDE.
		We shall see that using boundary measurements alone, (\ref{S1E23}) and (\ref{S1E24}), we are only able to reconstruct the potential $\phi$ within an ``infinitesimal'' neighbourhood of $\partial\Omega$, see \Cref{S2R2}, and that it is not possible to reconstruct $\phi$ at interior points unless one takes into account additional measurements, such as interior temperature measurements (\ref{S1E25}).

		\section{Main results}
		\label{Section3}
		
        For any subset $U\subset\mathbb{R}^N$ and $0\leq \alpha\leq 1$, we denote  the space of $\alpha$-H\"{o}lder continuous functions by $C^{0,\alpha}(U)$, with the convention $C^{0,0}(U)\equiv C^0(U)$. We denote the space of functions in $C^{0,\alpha}(U)$ with finite $\alpha$-H\"{o}lder-norm by $C^{0,\alpha}_{b}(U)$.
		If $U$ is open, we set 
        \begin{gather}
            \nonumber
        \dot{C}^1_b(\overline{U}):=\left\{f\in C^1(U)\mid \sup_{x\in U}|\nabla f(x)|<\infty\text{ and }\nabla f\text{ extends continuously to }\overline{U}\right\}\text{ and }
        \\
        \nonumber
        \dot{C}_b^{1,1}(\overline{U}):=\left\{f\in \dot{C}^1_b(\overline{U})\mid \sup_{x,y\in U,x\neq y}\frac{|\nabla f(x)-\nabla f(y)|}{|x-y|}<\infty\right\}.
        \end{gather}

Throughout the paper, we assume the following ellipticity conditions in $\Omega$:
\begin{equation}\label{eq:ellipt}
\epsilon \ge \lambda,\qquad D_i\ge \lambda\;\text{for every $i=1,\dots,M$,}\qquad \partial_s\phi\ge\lambda    
\end{equation}
for some $\lambda>0$, where $\phi=\phi(p,s,x)$, $(p,s,x)\in \mathbb{R}^{M+1}\times \Omega$. In the following, we consider the quantity $\phi(c(x),T(x),x)$, so that we essentially demand that the potential is monotonically increasing as a function of the temperature (see Section~\ref{subssec:functional_assumptions} for some motivations on this condition).
Upon applying the chain rule in (\ref{S1E20}) to $\phi(c(x),T(x),x)$, it becomes clear that the assumptions in (\ref{eq:ellipt}) make our system of equations elliptic.

        \subsection{The forward problem}
        Recall that we consider the static equations
		\begin{gather}
			\label{S2E1}
			-\operatorname{div}(D_i(c(x),T(x),x)\nabla c_i(x))=g_i(c(x),T(x),x)\text{ and }\operatorname{div}(\epsilon(x) \nabla (\phi(c(x),T(x),x))=q\cdot c(x).
		\end{gather}
        The following result guarantees the existence of weak solutions for the forward problem.
        \begin{thm}[Existence of weak solutions]
			\label{S3C4}
			Let $\Omega\subset\mathbb{R}^3$ be a bounded $C^1$-domain, $M\in \mathbb{N}$, $q\in \mathbb{R}^M$, $\phi\in \dot{C}_b^1(\mathbb{R}^{M+1}\times \overline{\Omega})$, $D_i\in C^0_b(\mathbb{R}^{M+1}\times \overline{\Omega})$, $g_i\in C^0_b(\mathbb{R}^{M+1}\times \overline{\Omega})$ for $1\leq i\leq M$, $\epsilon\in L^{\infty}(\Omega)$. Suppose further that  $(\epsilon,D_i,\phi)$ satisfy the ellipticity condition (\ref{eq:ellipt}). Then for every $(\gamma,\tau)\in \left(W^{\frac{1}{2},2}(\partial\Omega)\right)^{M+1}$ there exists at least one weak solution $(c,T)\in \left(H^1(\Omega)\right)^{M+1}$ of the following boundary value problem
			\begin{gather}
				\label{S3E9}
				-\operatorname{div}(D_i(c,T,x)\nabla c_i)=g_i(c,T,x)\text{ and }\operatorname{div}(\epsilon \nabla (\phi(c,T,x)))=q\cdot c\text{ in }\Omega\text{, }(c,T)|_{\partial\Omega}=(\gamma,\tau),
			\end{gather}
			where $c=(c_1,\dots,c_M)$, $1\leq i\leq M$.
		\end{thm}
The proof of \Cref{S3C4} consists of two steps. First, we rewrite the system in an equivalent form, which is more symmetric in the unknowns $D_i$ and $\phi$; second, we use a standard Schaefer fix point argument to show that this equivalent system of equations admits a weak solution, cf.\ \Cref{Existence}.

The uniqueness of solutions is more subtle. We obtain the following partial result regarding the uniqueness of weak solutions, cf.\ \Cref{Uniqueness} for the proof; see also \Cref{S3R6} for an example of a related system of equations where uniqueness fails.
        \begin{thm}[Uniqueness of source-free weak solutions with constant boundary conditions]
			\label{S3P5}
			Let $\Omega\subset\mathbb{R}^3$ be a bounded domain with $C^1$-boundary and $M\in \mathbb{N}$. Let $D_i\in C^{0}_b(\mathbb{R}^{M+1}\times \overline{\Omega})$ for $1\leq i\leq M$, $q\in \mathbb{R}^M$, $\epsilon\in L^{\infty}(\Omega)$ and $\phi\in \dot{C}_b^1(\mathbb{R}^{M+1}\times \overline{\Omega})$. Suppose further that $(\epsilon,D_i,\phi)$ satisfy the ellipticity condition (\ref{eq:ellipt}). Then for every $(\gamma,\tau)\in \mathbb{R}^M\times W^{\frac{1}{2},2}(\partial\Omega)\subset \left(W^{\frac{1}{2},2}(\partial\Omega)\right)^{M+1}$ there is a unique weak solution $(c,T)\in \left(H^1(\Omega)\right)^{M+1}$ solving the boundary value problem
			\begin{gather}
				\label{S3E11}
				\operatorname{div}\left(D_i(c,T,x)\nabla c_i\right)=0\text{ and }\operatorname{div}(\epsilon \nabla (\phi(c,T,x)))=q\cdot c\text{ in }\Omega\text{, }(c,T)|_{\partial\Omega}=(\gamma,\tau).
			\end{gather}
			In particular, $c=\gamma$ throughout $\Omega$.
		\end{thm} 
        The key observation in the proof of \Cref{S3P5} is that constant functions $\gamma_i$ satisfy $\nabla \gamma_i=0$, and thus they satisfy the PDE $\operatorname{div}(D_i(\gamma,T,x)\nabla \gamma_i)=0$ and the corresponding constant boundary conditions.
        \subsection{The inverse problem}
		\subsubsection{Boundary measurements}
		Let us first discuss how to formally define boundary measurements.

		If the sources $g_i$ are square integrable we conclude from (\ref{S2E1}) that $\operatorname{div}(D_i\nabla c_i)\in L^2(\Omega)$ and $D_i\nabla c_i\in L^2(\Omega)$, so that we can make sense of the normal traces $\mathcal{N}\cdot \left(D_i\nabla c_i\right)\in W^{-\frac{1}{2},2}(\partial\Omega)$, c.f.\ \cite[I \S 2 Theorem 2.5]{GR86}.
        The temperature is of class $H^1(\Omega)$ as a weak solution so that $\nabla T\in L^2(\Omega)$. It follows further from (\ref{S2E1}), by means of the chain rule, that under the assumptions of the upcoming \Cref{S2T1} we have $\operatorname{div}(\nabla T)\in L^1(\Omega)$. Consequently, we can similarly make sense of the normal trace $\mathcal{N}\cdot \nabla T$ as an element of $W^{-\frac{1}{p},p}(\partial\Omega)$, the topological dual space of $W^{1-\frac{1}{p},p}(\partial\Omega)$, for any $p>3$.
        
        This allows us to consider the corresponding Cauchy-data of the solutions associated with (\ref{S2E1})
		\begin{gather}
			\label{S2E2}
			\mathcal{C}:=\{(c|_{\partial\Omega},T|_{\partial\Omega},\mathcal{N}\cdot D_1(c,T,x)\nabla c_1,\dots,\mathcal{N}\cdot (D_M(c,T,x)\nabla c_M),\mathcal{N}\cdot \nabla T)\mid (c,T)\text{ solves (\ref{S2E1})}\},
		\end{gather}
		with $\mathcal{C}\subset \bigl(W^{\frac{1}{2},2}(\partial\Omega)\bigr)^{M+1}\times \bigl(W^{-\frac{1}{2},2}(\partial\Omega)\bigr)^M\times W^{-\frac{1}{p},p}(\partial\Omega)$ for any $p>3$.
        
        The collection of Cauchy-data depends on $D_i,g_i,\phi,q$ and $\epsilon$. Whenever of relevance, we indicate such a dependence by writing, for example, $\mathcal{C}[\phi]$. It is not difficult to see that if $D_i\in C^{0,1}(\overline{\Omega})$ and $g_i\in C^0(\overline{\Omega})$ depend on position alone, then there exists a unique solution of (\ref{S2E1}) for any prescribed Dirichlet boundary conditions. Therefore, in this case, we may view the Cauchy-data as a Dirichlet-to-Neumann map $$\Lambda_{\operatorname{DN}}\colon\bigl(W^{\frac{1}{2},2}(\partial\Omega)\bigr)^{M+1}\rightarrow \bigl(W^{-\frac{1}{2},2}(\partial\Omega)\bigr)^M\times W^{-\frac{1}{p},p}(\partial\Omega).$$

        The following result shows that boundary measurements alone cannot reconstruct the electric potential.
		\begin{thm}[Necessity of interior measurements]
			\label{S2T1}
			Let $\Omega\subset\mathbb{R}^3$ be a bounded domain with connected $C^1$-boundary, let $M\in \mathbb{N}$ and for $1\leq i\leq M$, $g_i\in C^0(\overline{\Omega})$, $D_i\in C^{0,1}(\overline{\Omega})$, $\epsilon\in C^{0,1}(\overline{\Omega})$, $\phi^1,\phi^2\in \dot{C}^{1,1}_b(\mathbb{R}^{M+1}\times \overline{\Omega})$ and $q\in \mathbb{R}^M$. Further, suppose that $(\epsilon,D_i,\phi^1)$ and $(\epsilon,D_i,\phi^2)$ satisfy the ellipticity condition (\ref{eq:ellipt}). Then the following two sets of properties are equivalent:
			\begin{enumerate}
				\item
				\begin{enumerate}
					\item For all $(\gamma,\tau)\in \left(W^{\frac{1}{2},2}(\partial\Omega)\right)^{M+1}$ we have $$\phi^1(\gamma(x),\tau(x),x)-\phi^1(\gamma(y),\tau(y),y)=\phi^2(\gamma(x),\tau(x),x)-\phi^2(\gamma(y),\tau(y),y)$$
                    for $\mathcal{H}^2$-a.e.\ $x,y\in \partial\Omega$;
					\item $\Lambda_{\operatorname{DN}}[\epsilon,D_i,g_i,\phi^1]=\Lambda_{\operatorname{DN}}[\epsilon,D_i,g_i,\phi^2]$.
				\end{enumerate}
				\item \begin{enumerate}
					\item $\nabla_x(\phi^1(p,s,x)-\phi^2(p,s,x))=0$ for all $(p,s,x)\in \mathbb{R}^{M+1}\times \partial\Omega$ where $\nabla_x$ denotes the full Euclidean gradient with respect to $x$;
					\item there is some $r\in \mathbb{R}$ such that for all $(p,s,x)\in \mathbb{R}^{M+1}\times \partial\Omega$ we have $\phi^1(p,s,x)-\phi^2(p,s,x)=r$.
				\end{enumerate}
			\end{enumerate}
		\end{thm}
        The idea of the proof of \Cref{S2T1} consists of two steps. In the first step, one can differentiate the boundary identity along the boundary and show, by an appropriate choice of boundary conditions, that the corresponding gradient of $\phi^1-\phi^2$ must vanish, as a function on $\mathbb{R}^{M+1}\times \partial\Omega$, cf.\ \Cref{S3L7}. This establishes the equivalence between (i,a) and (ii,a). The remaining two bullet points are shown to be equivalent by exploiting the fact that $D_i$ and $g_i$ are functions of position alone and therefore the equations decouple. This enables the use of standard uniqueness results to establish a relationship between the normal traces of the temperature gradients and the normal traces of the gradients of the $\phi^j$, cf.\ \Cref{ProofBoundary}.
        \begin{rem}
			\label{S2R2}
			\begin{enumerate}
				\item \Cref{S2T1} tells us that boundary measurements alone are enough to reconstruct the values of the potential along $\partial\Omega$ (up to a constant) and in addition, we can reconstruct the full gradient of the potential along $\partial\Omega$. This means that while we cannot reconstruct the potential $\phi$ at any fixed interior point, the behaviour of $\phi$ as we approach the boundary is determined through the gradient of $\phi$ by the boundary measurements. In this sense boundary measurements can reconstruct the values of $\phi$ along the boundary and in an ``infinitesimal'' neighbourhood around it.
				\item If $\phi^1$ is a potential, $\psi\in C^{\infty}_c(\Omega)$ is a bump function and $\widetilde{\phi}\in \dot{C}^{1,1}_b(\mathbb{R}^{M+1}\times \overline{\Omega})$ is any other function, then $\phi^2:=\phi^1+\psi\cdot \widetilde{\phi}$ will satisfy conditions (ii,a) and (ii,b), from which we can conclude that their Dirichlet-to-Neumann maps coincide. This shows that, in general, no information about the behaviour of the potential at interior points can be obtained from boundary measurements alone.
			\end{enumerate}
		\end{rem}
		\subsubsection{Boundary and internal measurements}
		The goal of the present section is to show that when interior temperature measurements are taken into account, then the potential as well as the diffusion coefficients can be uniquely reconstructed, at least in the source-free setting.
		
		In the following we will consider the static source-free problem
		\begin{gather}
			\label{S2E3}
			\operatorname{div}(D_i(c,T,x)\nabla c_i)=0\text{ and }\operatorname{div}(\epsilon\nabla (\phi(c,T,x))=q\cdot c\text{ in }\Omega.
		\end{gather}
		The idea is to replace the temperature boundary measurements $\mathcal{N}\cdot \nabla T$ by interior temperature measurements. The boundary measurements of interest are therefore captured by the following (reduced) Cauchy-data set
		\begin{gather}
			\label{S2E4}
            \mathcal{C}_{\operatorname{red}}:=\{(c|_{\partial\Omega},T|_{\partial\Omega},\mathcal{N}\cdot D_1(c,T,x)\nabla c_1,\dots, \mathcal{N}\cdot D_M(c,T,x)\nabla c_M\mid (c,T)\text{ solves (\ref{S2E3}}\},
		\end{gather}
		which is a subset of $\left(W^{\frac{1}{2},2}(\partial\Omega)\right)^{M+1}\times \left(W^{-\frac{1}{2},2}(\partial\Omega)\right)^M$. To indicate the dependence of $\mathcal{C}_{\operatorname{red}}$ on the $D_i$ and $\phi$ we will also write $\mathcal{C}_{\operatorname{red}}[D,\phi]$ where we use here the notation $D=(D_1,\dots,D_M)$.
        
        The main result of this subsection is the following.
		\begin{thm}[Unique reconstruction for the inverse problem with boundary and internal data]
			\label{S2T3}
			Let $\Omega\subset\mathbb{R}^3$ be a bounded domain with a connected $C^{1,\beta}$-boundary for some $0<\beta\leq 1$. Let $M\in \mathbb{N}$ and for $1\leq i\leq M$ suppose $D^1_i,D^2_i\in C^{0,1}_b(\mathbb{R}^{M+1}\times \overline{\Omega})$, $\epsilon\in C^{0,\beta}(\overline{\Omega})$, $\phi^1,\phi^2\in \dot{C}^1_b(\mathbb{R}^{M+1}\times \overline{\Omega})$ and $q\in \mathbb{R}^M$. Further, suppose that $(\epsilon,D^1_i,\phi^1)$ and $(\epsilon,D^2_i,\phi^2)$ satisfy the ellipticity condition (\ref{eq:ellipt}). Then the following two sets of conditions are equivalent
			\begin{enumerate}
				\item \begin{enumerate}
					\item For all $(\gamma,\tau)\in \left(W^{\frac{1}{2},2}(\partial\Omega)\right)^{M+1}$ and $\mathcal{H}^2$-a.e. $x,y\in \partial\Omega$ we have $$\phi^1(\gamma(x),\tau(x),x)-\phi^1(\gamma(y),\tau(y),y)=\phi^2(\gamma(x),\tau(x),x)-\phi^2(\gamma(y),\tau(y),y);$$
					\item For all $(\gamma,\tau)\in \left( W^{\frac{1}{2},2}(\partial\Omega)\right)^{M+1}$ there is some $\alpha=\alpha(\gamma,\tau)\in \left(W^{-\frac{1}{2},2}(\partial\Omega)\right)^M$ with $(\gamma,\tau,\alpha)\in \mathcal{C}_{\operatorname{red}}[D^1,\phi^1]\cap\mathcal{C}_{\operatorname{red}}[D^2,\phi^2]$;
					\item For all $(\gamma,\tau)\in \mathbb{R}^M\times W^{\frac{1}{2},2}(\partial\Omega)\subset \left(W^{\frac{1}{2},2}(\partial\Omega)\right)^{M+1}$ we have $T^1(x)=T^2(x)$ for a.e.\ $x\in \Omega$ where $T^1$ and $T^2$ are the unique temperature profiles obtained from the Dirichlet boundary condition $(\gamma,\tau)$ and (\ref{S2E3}), corresponding to the coefficient functions $(D^1_i,\phi^1,\epsilon,q)$ and $(D^2_i,\phi^2,\epsilon,q)$ respectively, cf.\ \Cref{S3P5}.
				\end{enumerate}
				\item  \begin{enumerate}
					\item For all $(p,s,x)\in \mathbb{R}^{M+1}\times \overline{\Omega}$ and all $1\leq i\leq M$ we have $D^1_i(p,s,x)=D^2_i(p,s,x)$,
					\item There is some $r\in \mathbb{R}$ such that for all $(p,s,x)\in \mathbb{R}^{M+1}\times \overline{\Omega}$ we have $\phi^1(p,s,x)-\phi^2(p,s,x)=r$.
				\end{enumerate}
			\end{enumerate}
		\end{thm}
        The key steps to establish \Cref{S2T3} are the following; see \Cref{InteriorMeasure} for the full proof.
        \begin{itemize}
            \item In the first step we show that we can express the temperature $T$ as a function of the concentrations $c$, (\ref{S3E30}). This dependence is non-local and non-linear.
            \item In the second step we generalise a linearisation technique of Sun, cf.\ \cite{Sun96}, to the setting of quasi-linear elliptic systems which may involve a non-local, non-linear dependence. In the original context of \cite{Sun96}, which considered a quasi-linear, scalar Calder\'{o}n problem, the linearisation corresponds to freezing the non-linearity and reduces the problem to the standard linear Calder\'{o}n problem. Due to the non-local dependence, our situation turns out to be more complicated. The non-linearity turns out to be a function of position after the linearisation procedure.
            \item We use the interior measurements to freeze the non-linearity and recover the potential and diffusion coefficients.
        \end{itemize}
		\begin{rem}
			\label{S2R4}
			It is well-known that the electrostatic (scalar) potential $\phi$ is determined only up to constants, so that \Cref{S2T3} is the best possible result in the sense that it shows that the diffusion coefficients are uniquely determined by our hybrid measurements and that the potential can be recovered up to a constant.
		\end{rem}

\subsection{Comparison to other inverse problems in the mathematical literature}

The problem of reconstructing the diffusion coefficients $D_i(c, T, x)$ and the electric potential $\phi(c, T, x)$ from boundary measurements and internal temperature data is at the intersection of three very active fields: the classical Calder\'on problem, inverse problems for quasilinear elliptic equations, and hybrid (coupled-physics) inverse problems.

\paragraph{The Classical Calder\'on Problem.} The linear prototype for the electrostatic sub-problem is Calder\'on's problem \cite{calderon1980inverse}, which aims to recover a spatially varying conductivity from the Dirichlet-to-Neumann (DN) map. Global uniqueness in dimensions $n \ge 3$ was established by Sylvester and Uhlmann using Complex Geometric Optics (CGO) solutions \cite{sylvester1987global}. Subsequent research relaxed regularity assumptions \cite{nachman1996global, astala2006calderon, haberman2015uniqueness} and addressed partial boundary data \cite{kenig2007calderon}. However, boundary-only measurements for elliptic equations inherently suffer from logarithmic stability \cite{alessandrini1988stable}, necessitating internal data or strong prior assumptions for practical reconstructions.

\paragraph{Quasilinear and Nonlinear Elliptic PDEs.} Because $D_i$ and $\phi$ depend on the state variables $c$ and $T$, the system governing the electrolyser is inherently quasilinear. The foundational uniqueness result for recovering solution-dependent coefficients from the DN map was established by Sun \cite{Sun96}, who introduced the method of linearisation at constant boundary data. This was later extended to anisotropic and gradient-dependent media \cite{sun1997inverse, hervas2002inverse}. Recently, the paradigm of higher-order linearisation has demonstrated that nonlinearities can actually facilitate uniqueness. C\^arstea et al.\ \cite{carstea2019reconstruction} developed constructive methods for recovering quasilinear coefficients by exploiting the nonlinear interactions of CGO solutions, and these results have recently been expanded to partial data settings \cite{kian2023quasilinear}.

\paragraph{Hybrid Inverse Problems and Coupled Systems.} To overcome the severe ill-posedness of boundary measurements, hybrid inverse problems utilise internal functionals (e.g., power or current densities). Seminal works by Bal and Uhlmann \cite{bal2010inverse} demonstrated that internal data transforms the elliptic inverse problem into a well-posed transport problem, yielding Lipschitz stability. In our formulation, the internal temperature $T(x)$ naturally serves as this internal data. Knowing the solution to one of the coupled equations everywhere inside the domain provides powerful constraints, conceptually linking our setup to inverse source and heat conduction problems with internal sensors \cite{bal2013hybrid,montalto2013stability,kuchment2012stabilizing}. 

Finally, the primary reaction-diffusion-Poisson model considered here is very close to Poisson-Nernst-Planck (PNP) systems, a connection that becomes explicit when electric forces are incorporated into the particle flux (as discussed in Appendix B). Inverse problems in the coupled electrochemical context frequently aim to recover specific diffusion parameters or doping profiles from boundary data \cite{burger2001identification,burger2007inverse, burger2011inverse}. Our work synthesises these analytical domains: we leverage the analytical tools of nonlinear elliptic theory (via linearisation) while exploiting the structural constraints provided by hybrid internal measurements to overcome the fundamental non-uniqueness of the boundary-only problem, offering a robust framework for identifying solution-dependent parameters in coupled systems.
        
		\section{Proofs}
        \label{Section4}
		\subsection{The forward problem}
		\label{Forward}
		\subsubsection{Existence of weak solutions}
		\label{Existence}
		We start by examining a simplified system of equations. Once this is done we will prove that our original system of equations can be reduced to this simplified system.
		\begin{lem}[Existence of weak solutions I]
			\label{S3L1}
			Let $\Omega\subset\mathbb{R}^3$ be a bounded domain with $C^1$-boundary (not necessarily connected). Let $N\in \mathbb{N}$, $\nu_i\in C^0_b(\mathbb{R}^{N}\times \overline{\Omega})$ for $1\leq i\leq N$, $G_i\in C^0_b(\mathbb{R}^{N}\times \overline{\Omega})$ for $1\leq i\leq N-1$ and $G_N\in C^0(\mathbb{R}^{N}\times \overline{\Omega})$. Suppose further that there are $0<\nu_*,C_1,C_2$ such that $\nu_*\leq \nu_i$ for all $1\leq i \leq N$ and $|G_N(p,x)|\leq C_1+C_2|\hat{p}|$ for all $(p,x)\in \mathbb{R}^N\times \overline{\Omega}$ where $p=(\hat{p},p_N)\in \mathbb{R}^{N-1}\times \mathbb{R}$. Then for every $\gamma\in \bigl(W^{\frac{1}{2},2}(\partial\Omega)\bigr)^{N}$ there exists some $\eta\in \left(H^1(\Omega)\right)^{N}$, $\eta=(\eta_1,\dots,\eta_N)$, that weakly solves for every $i=1,\dots,N$
			\begin{gather}
				\label{S3E1}
				-\operatorname{div}(\nu_i(\eta,x)\nabla \eta_i(x))=G_i(\eta(x),x)\text{ in }\Omega, \qquad\eta|_{\partial\Omega}=\gamma.
			\end{gather}
		\end{lem}
		\begin{proof}
        Let $\gamma=(\gamma_1,\dots,\gamma_N)\in \bigl(W^{\frac{1}{2},2}(\partial\Omega)\bigr)^N$. We split the proof into several steps.
        
			\textit{Step 1: The solution to \eqref{S3E1} as a fixed point of a suitable operator $A$.}  By standard linear elliptic  theory, for every $v\in \left(L^2(\Omega)\right)^N$ there exist unique weak solutions $w_i\in H^1(\Omega)$, $1\leq i\leq N$, of the boundary value problem
			\begin{gather}
				\label{S3E2}
				-\operatorname{div}(\nu_i(v(x),x)\nabla w_i(x))=G_i(v(x),x)\text{ in }\Omega, \qquad w_i|_{\partial\Omega}=\gamma_i.
			\end{gather}
			We can therefore consider the solution operator
			\begin{gather}
				\label{S3E3}
				A:\left(\left(L^2(\Omega)\right)^N,\|\cdot\|_{L^2(\Omega)}\right)\rightarrow \left(\left(H^1(\Omega)\right)^N,\|\cdot\|_{L^2(\Omega)}\right),
			\end{gather}
			which maps $v\in \left(L^2(\Omega)\right)^N$ to the unique solution $w=(w_1,\dots,w_N)$ of (\ref{S3E2}). We observe that solutions of (\ref{S3E1}) correspond to fix points of $A$. We verify now that the operator $A$ satisfies the conditions of Schaefer's fixed point theorem, cf.\ \cite[Chapter 9.2.2 Theorem 4]{Evans10}: we have to show that $A$ is continuous, compact and that the set $\left\{v\in \left(L^2(\Omega)\right)^N\mid v=\lambda A(v)\text{ for some }0\leq\lambda\leq 1\right\}$ is bounded.
			
			\textit{Step 2: $A$ is continuous.} Let $v,\widetilde{v}\in \left(L^2(\Omega)\right)^N$ with associated solutions $w=A(v),\widetilde{w}=A(\widetilde{v})$. Let us set $\overline{w}:=w-\widetilde{w}$ (and denote the components by $w_i,\widetilde{w}_i$, and $\overline{w}_i$, respectively). By (\ref{S3E2}), we have
			\begin{equation*}
				\begin{split}
                \int_{\Omega}&\nu_i(v,x)\nabla \overline{w}_i\cdot \nabla \overline{w}_idx=\int_{\Omega}\nu_i(v,x)\nabla w_i\cdot \nabla \overline{w}_idx-\int_{\Omega}\nu_i(v,x)\nabla \widetilde{w}_i\cdot \nabla \overline{w}_idx
				\\
				&=\int_{\Omega}\nu_i(v,x)\nabla w_i\cdot \nabla \overline{w}_idx-\int_{\Omega}\nu_i(\widetilde{v},x)\nabla \widetilde{w}_i\cdot \nabla \overline{w}_idx+\int_{\Omega}\left(\nu_i(\widetilde{v},x)-\nu_i(v,x)\right)\nabla \widetilde{w}_i\cdot \nabla \overline{w}_idx
				\\
				&=\int_{\Omega}(G_i(v,x)-G_i(\widetilde{v},x))\overline{w}_idx+\int_{\Omega}(\nu_i(\widetilde{v},x)-\nu_i(v,x))\nabla \widetilde{w}_i\cdot \nabla \overline{w}_idx,
                \end{split}
			\end{equation*}
			where we used that $\overline{w}_i\in H^1_0(\Omega)$. Using that $\nu_*\leq \nu_i$ by assumption, the Cauchy-Schwarz inequality and Poincar\'{e}'s inequality, we find
			\begin{gather}
				\label{S3E4}
				\nu_*\|\nabla \overline{w}_i\|_{L^2(\Omega)}\leq \|(\nu_i(\widetilde{v},\cdot)-\nu_i(v,\cdot))\nabla \widetilde{w}_i\|_{L^2(\Omega)}+\frac{\|G_i(v,\cdot)-G_i(\widetilde{v},\cdot)\|_{L^2(\Omega)}}{\sqrt{\lambda_D(\Omega)}},
			\end{gather}
			where $\lambda_D(\Omega)$ denotes the first Dirichlet-eigenvalue of $\Omega$. 
            
            Let $(v_k)_k\subset \left(L^2(\Omega)\right)^N$ be a sequence converging to $\widetilde{v}$ in $L^2(\Omega)$. We can then pick $v=v_k$ in (\ref{S3E4}) and obtain
			\begin{gather}
				\label{S3E5}
				\nu_*\|\nabla \widetilde{w}_i-\nabla w^k_i\|_{L^2(\Omega)}\leq \|(\nu_i(\widetilde{v},\cdot)-\nu_i(v_k,\cdot))\nabla \widetilde{w}_i\|_{L^2(\Omega)}+\frac{\|G_i(v_k,\cdot)-G_i(\widetilde{v},\cdot)\|_{L^2(\Omega)}}{\sqrt{\lambda_D(\Omega)}}.
			\end{gather}
			We will now show that every subsequence of $(v_k)_k$ has yet another subsequence such that its image under $A$ converges to $A(\tilde{v})$. Since the limit is independent of the chosen subsequence, this will prove the convergence of the full sequence $(A(v_k))_k$ to $A(\tilde{v})$. To this end, observe that every subsequence of $(v_k)_k$ has yet another subsequence (denoted in the same way) for which $v_k(x)\rightarrow \widetilde{v}(x)$ pointwise for a.e. $x\in \Omega$ so that the continuity of $\nu_i$ and its global boundedness imply that $\|(\nu_i(\widetilde{v},\cdot)-\nu_i(v_k,\cdot))\nabla \widetilde{w}_i\|_{L^2(\Omega)}\rightarrow 0$ as $k\rightarrow\infty$ by means of the dominated convergence theorem. Further, by means of the continuity of the $G_i$ we also find $G_i(v_k(x),x)-G_i(\widetilde{v}(x),x)\rightarrow 0$ for a.e. $x\in \Omega$ and in addition, by assumption on $G_i$, $|G_i(v_k(x),x)-G_i(\widetilde{v}(x),x)|\leq 2C_1+C_2(|v_k(x)|+|\widetilde{v}(x)|)$. The latter, converges pointwise a.e.\ and in $L^2(\Omega)$ to $2C_1+2C_2|\widetilde{v}(x)|$ so that the generalised dominated convergence theorem, \cite[Chapter 2.3 Exercise 20]{Folland99}, implies $G_i(v_k(x),x)-G_i(\widetilde{v}(x),x)\rightarrow 0$ in $L^2(\Omega)$ and consequently (\ref{S3E5}) implies
			\begin{gather}
				\nonumber
				\|\nabla \widetilde{w}_i-\nabla w^k_i\|_{L^2(\Omega)}\rightarrow 0\text{ as }k\rightarrow 0\text{ for all }1\leq i\leq N.
			\end{gather}
			Then the Poincar\'{e} inequality yields $\|A(v_k)-A(\widetilde{v})\|_{L^2(\Omega)}=\|\widetilde{w}-w_k\|_{L^2(\Omega)}\rightarrow 0$, which implies the continuity of $A$.
			
			\textit{Step 3: $A$ is compact.} Let $v\in \left(L^2(\Omega)\right)^N$, $w:=A(v)$, and $w^0:=A(0)$. By \eqref{S3E4}, we have
			\begin{gather}
				\nonumber
				\nu_*\|\nabla w_i-\nabla w^0_i\|_{L^2(\Omega)}\leq C(\nu)\|\nabla w^0_i\|_{L^2(\Omega)}+\widetilde{C}(\Omega,G)(1+\|v\|_{L^2(\Omega)})
			\end{gather}
			for suitable constants $C(\nu)$ and $\widetilde{C}(\Omega,G)$ that are independent of $v$. Here we used the bound on the $G_i$ and that the $\nu_i$ are globally bounded. From this we deduce by means of Poincar\'{e}'s inequality
			\begin{gather}
				\label{S3E6}
				\|w-w^0\|_{H^1(\Omega)}\leq c(\Omega,\nu,G,\gamma)\left(1+\|v\|_{L^2(\Omega)}\right)
			\end{gather}
			and in turn $\|w\|_{H^1(\Omega)}\leq \|w^0\|_{H^1(\Omega)}+\|w-w^0\|_{H^1(\Omega)}$. This proves that if $(v_k)_k\subset \left(L^2(\Omega)\right)^N$ is $L^2(\Omega)$-bounded, then $(A(v_k))_k$ is bounded in $H^1(\Omega)$ and thus admits a subsequence that converges strongly in $L^2(\Omega)$. This establishes the compactness of $A$.

            \textit{Step 4: the set $\left\{v\in \left(L^2(\Omega)\right)^N\mid v=\lambda A(v)\text{ for some }0\leq\lambda \leq 1\right\}$ is $L^2(\Omega)$-bounded.} Assume that $v=\lambda A(v)$ for some $0\leq \lambda\leq 1$. Observe that for $1\leq i\leq N-1$ the $G_i$ are globally bounded, so that (\ref{S3E4}) immediately implies that $\|\nabla w_i-\nabla w^0_i\|_{L^2(\Omega)}\leq C(\nu,G,\Omega,\gamma)$, where $w=A(v)$. So, by Poincar\'{e}'s inequality, we deduce that $\|w_i\|_{H^1(\Omega)}$ is uniformly bounded for $1\leq i\leq N-1$ (independently of the $L^2(\Omega)$-norm of $v$). Finally, for $i=N$, we can deduce from (\ref{S3E4}) that
			\begin{gather}
				\label{S3E7}
				\nu_*\|\nabla w_N-\nabla w^0_N\|_{L^2(\Omega)}\leq C(\nu,\gamma,\Omega)+C_1(\Omega,G,\gamma)\left(1+\|\hat{v}\|_{L^2(\Omega)}\right)
			\end{gather}
			where $v=(\hat{v},v_N)\in \left(L^2(\Omega)\right)^{N-1}\times L^2(\Omega)$. Since $\hat{v}=\lambda \hat{w}$, where $w=(\hat w,w_N)$, we have $\|\hat{v}\|_{L^2(\Omega)}\leq \|\hat{w}\|_{L^2(\Omega)}\leq \sqrt{N-1}\max_{1\leq i\leq N-1}\|w_i\|_{L^2(\Omega)}$ and the latter is uniformly bounded. We deduce from (\ref{S3E7}) and Poincar\'{e}'s inequality that $\|w_N\|_{H^1(\Omega)}$ is also uniformly bounded. As a consequence, the set $$ \left\{v\in \left(L^2(\Omega)\right)^N\mid v=\lambda A(v)\text{ for some }0\leq \lambda\leq 1\right\}$$ is bounded. 
            
            \textit{Step 5: \eqref{S3E1} has a weak solution.}
            The existence of a weak solution to \eqref{S3E1} now follows from Schaefer's fixed point theorem.
		\end{proof}
        \begin{rem}
				\label{S3R2}
				We notice that for steps 1-3 of the proof, we only used that $\nu_i\in C^0_b(\mathbb{R}^N\times \overline{\Omega})$, $0<\nu_*\leq \nu_i$ for all $1\leq i\leq N$ and that $G_i\in C^0(\mathbb{R}^N\times \overline{\Omega})$ with $|G_i(p,x)|\leq C_1+C_2|p|$ for all $1\leq i\leq N$. Thus, the continuity and the compactness of $A$ remain valid under these more general assumptions.
			\end{rem}
		Before we can exploit \Cref{S3L1} in order to establish existence of solutions to our boundary value problem (\ref{S2E1}) we need the following result
		\begin{lem}
			\label{S3L3}
			Let $\Omega\subset\mathbb{R}^3$ be a bounded $C^1$-domain, $M\in \mathbb{N}$, and $\phi\in \dot{C}^1_b(\mathbb{R}^{M+1}\times \overline{\Omega})$. Suppose that $\phi$ satisfies the ellipticity condition in (\ref{eq:ellipt}). Then there exists $h\in \dot{C}_b^1(\mathbb{R}^{M+1}\times \overline{\Omega})$ such that 
            \[\phi(p,h(p,s,x),x)=s=h(p,\phi(p,s,x),x),\qquad (p,s,x)\in \mathbb{R}^{M+1}\times \overline{\Omega}.
            \]
		\end{lem}
		\begin{proof}
			\textit{Step 1: Existence of a continuous inverse $h$.} For fixed $(p,x)\in \mathbb{R}^M\times \overline{\Omega}$, the map $s\mapsto \phi(p,s,x)$ is strictly increasing and maps $\mathbb{R}$ onto $\mathbb{R}$ (since $0<\lambda\leq (\partial_s\phi)(p,s,x)$). Consequently, there is some function $h\colon \mathbb{R}^{M+1}\times \overline{\Omega}\to\mathbb R$ such that   $\phi(p,h(p,s,x),x)=s=h(p,\phi(p,s,x),x)$ for all $(p,s,x)\in \mathbb{R}^{M+1}\times \overline{\Omega}$. Furthermore, for fixed $(p,x)$, $h(p,\cdot,x)$  is $C^1$ and
			\begin{gather}
				\label{S3E8}
				(\partial_sh)(p,s,x)=\frac{1}{(\partial_s\phi)(p,h(p,s,x),x)}.
			\end{gather}
			We first show that $h$ is continuous as a function of $(p,s,x)$. Let $(p_j,s_j,x_j)\in \mathbb{R}^{M+1}\times \overline{\Omega}$, $j=1,2$, and observe that by definition
			\begin{gather}
				\nonumber
				s_2-s_1=\phi(p_2,h(p_2,s_2,x_2),x_2)-\phi(p_1,h(p_1,s_1,x_1),x_1).
			\end{gather}
			Now set $t_j:=h(p_j,s_j,x_j)$ and $w_j:=(p_j,t_j,x_j)$. Then
			\begin{gather}
				\nonumber
				s_2-s_1=\phi(p_2,t_2,x_2)-\phi(p_1,t_1,x_1)=\phi(w_2)-\phi(w_1)=\int_0^1\frac{d}{d\sigma}\left(\phi(w_1+\sigma (w_2-w_1))\right) d\sigma
				\\
				\nonumber
				=(t_2-t_1)\int_0^1(\partial_t\phi)(w_1+\sigma (w_2-w_1))d\sigma+(p_2-p_1)\cdot \int_0^1(\nabla_p\phi)(w_1+\sigma (w_2-w_1))d\sigma
				\\
				\nonumber
				+(x_2-x_1)\cdot \int_0^1(\nabla_x\phi)(w_1+\sigma(w_2-w_1))d\sigma.
			\end{gather}
			We can solve the above equation for $(t_2-t_1)$ and note that since $0<\lambda\leq (\partial_t\phi)$ and since all derivatives of $\phi$ are uniformly bounded, we find $|t_2-t_1|\leq L(|s_2-s_1|+|p_2-p_1|+|x_2-x_1|)$. Lastly, recalling that $t_j=h(p_j,s_j,x_j)$, we conclude that $h$ is (globally Lipschitz) continuous on $\mathbb{R}^{M+1}\times \overline{\Omega}$.
            \newline
            \newline
            \textit{Step 2: Existence and continuity of first derivatives.} It follows first from (\ref{S3E8}) and the continuity of $\partial_s\phi$ and $h$ that $\partial_sh$ is also continuous.
			
			To prove the differentiability of $h$ with respect to $p$ and $x$ we use the implicit function theorem. Fix any $s_0\in \mathbb{R}, p_0\in \mathbb{R}^M, x_0\in \overline{\Omega}$ and set $t_0:=h(p_0,s_0,x_0)$. According to the implicit function theorem, there exist a neighbourhood $U$ of $(p_0,x_0)$, an open neighbourhood $I$ of $t_0$ and a $C^1$-function $g\colon U\rightarrow I$ such that if $(t,p,x)\in I\times U$ and $\phi(p,t,x)=s_0$, then $t=g(p,x)$. We observe now that $(t_0,p_0,x_0)\in I\times U$ and that $t_0=h(p_0,s_0,x_0)$. Then by continuity of $h$ the set $V:=h^{-1}(I)$ is an open neighbourhood of $(s_0,p_0,x_0)$ in $\mathbb{R}^{1+M}\times \overline{\Omega}$. In particular, there is some open interval $\widetilde{I}$ containing $s_0$ and some open neighbourhood $W\subset U$ of $(p_0,x_0)$ such that $\widetilde{I}\times W\subset V$. Then for every $(p,x)\in W$ we find $\phi(p,h(p,s_0,x),x)=s_0$, $h(p,s_0,x)\in I$, $(p,x)\in U$ and thus by the implicit function theorem we must have $h(p,s_0,x)=g(p,x)$, which is of class $C^1$ in the variables $p$ and $x$. This implies that the partial derivatives of $h$ w.r.t. $p$ and $x$ exist. We may then differentiate the expression $\phi(p,h(p,s,x),x)=s$ with respect to $p$ and $x$ and exploit the continuity of $h$ as a function of $(p,s,x)$ to deduce that $\nabla_xh$ and $\nabla_ph$ are both continuous functions of $(p,s,x)$.
            \newline
            \newline
            \textit{Step 3: Boundedness of the derivatives.} We have shown in step 1 that $h$ is globally Lipschitz continuous. Therefore all of its derivatives are globally bounded. Overall $h\in\dot{C}^1_b(\mathbb{R}^{M+1}\times \overline{\Omega})$.
		\end{proof}
		We are now ready to prove the existence of solutions for our forward problem (\ref{S2E1}).
		\begin{proof}[Proof of \Cref{S3C4}]
			In accordance with \Cref{S3L3}, we let $h$ denote the inverse of $\phi$ and define
            \begin{gather}
                \nonumber
                \nu_i(p,s,x):=D_i(p,h(p,s,x),x),\quad G_i(p,s,x):=g_i(p,h(p,s,x),x),\qquad 1\leq i\leq M.
            \end{gather}
            In addition, we let $G_{M+1}(p,s,x):=-p\cdot q$ and $\nu_{M+1}(p,s,x):=\epsilon(x)$. By \Cref{S3L1}, the following system of equations admits a weak solution $(c,\sigma)\in \left(H^1(\Omega)\right)^{M+1}$:
			\begin{equation*}
\left\{
\begin{array}{l}
     -\operatorname{div}(\nu_i(c(x),\sigma(x),x)\nabla c_i(x))=G_i(c(x),\sigma(x),x),\qquad 1\leq i\leq M,  \\
     -\operatorname{div}(\epsilon(x) \nabla \sigma(x))=G_{M+1}(c(x),\sigma(x),x),\\ 
     (c,\sigma)|_{\partial\Omega}=(\gamma,\phi(\gamma,\tau,\cdot)).
\end{array}
\right.                
			\end{equation*}
			We now set $T(x):=h(c(x),\sigma(x),x)\in H^1(\Omega)$ and observe that we have the identities for $1\leq i\leq M$
            \begin{align*}
        \nu_i(c(x),\sigma(x),x)&=D_i(c(x),h(c(x),\sigma(x),x),x)=D_i(c(x),T(x),x),
        \\
        G_i(c(x),\sigma(x),x)&=g_i(c(x),T(x),x).
            \end{align*}
            We deduce that, for $1\leq i\leq M$, we have
            \begin{gather}
                \nonumber
              -\operatorname{div}(D_i(c(x),T(x),x)\nabla c_i(x))=g_i(c(x),T(x),x)\;\text{ in }\Omega.  
            \end{gather}
            Further, $\phi(c(x),T(x),x)=\phi(c(x),h(c(x),\sigma(x),x),x)=\sigma(x)$ and therefore
            \begin{gather}
                \nonumber
             \operatorname{div}(\epsilon \nabla (\phi(c,T,x)))=\operatorname{div}(\epsilon \nabla \sigma)=\operatorname{div}(\nu_{M+1}\nabla \sigma)=c\cdot q\text{ in }\Omega.   
            \end{gather}
            Lastly, $T|_{\partial\Omega}=h(c|_{\partial\Omega},\sigma|_{\partial\Omega},x)=h(\gamma,\phi(\gamma,\tau,x),x)=\tau$, so that $(c,T)$ is a weak solution of (\ref{S3E9}), as desired.
		\end{proof}
		\subsubsection{Uniqueness of weak solutions}
		\label{Uniqueness}
		In this subsection we focus on the source-free system (\ref{S2E3}) with constant Dirichlet-boundary conditions. The uniqueness of these solutions will be used during the second part of the proof of \Cref{S2T3}.
		\begin{proof}[Proof of \Cref{S3P5}]
			We recall that, by assumption, $\gamma=(\gamma_1,\dots,\gamma_M)\in \mathbb{R}^M$ is a constant and observe first that we have the identity
            \begin{gather}
                \nonumber
                \operatorname{div}(D_i(\gamma(x),T(x),x)\nabla \gamma_i)=0\text{ in $\Omega$},\qquad T\in H^1(\Omega).
            \end{gather}
            Let us then define $\phi_0(x):=\phi(\gamma,\tau(x),x)\in W^{\frac{1}{2},2}(\partial\Omega)$, where $\tau\in W^{\frac{1}{2},2}(\partial\Omega)$ is the prescribed boundary condition for the temperature, recall (\ref{S3E11}). Further, let $L^{-1}_{\epsilon}$ denote the solution operator corresponding to the linear elliptic PDE
			\begin{gather}
				\label{S3E12}
				\operatorname{div}(\epsilon \nabla \eta)=v \text{ in }\Omega,\quad\eta|_{\partial\Omega}=\phi_0,
			\end{gather}
			so that $L^{-1}_{\epsilon}(v)=\eta$ is the unique weak $H^1(\Omega)$ solution of (\ref{S3E12}). Now, let $h$ denote the inverse of $\phi$ as in \Cref{S3L3}. It is then easy to see that
            \begin{gather}
                \nonumber
                (\gamma,\theta):=(\gamma, h(\gamma,L^{-1}_{\epsilon}(\gamma\cdot q)(x),x))
            \end{gather}
            provides a solution to (\ref{S3E11}). Suppose now that $(c,T)$ is any other solution to (\ref{S3E11}). Then, similarly as in the proof of \Cref{S3L1}, we find
			\begin{gather}
				\nonumber
				\int_{\Omega}D_i(c,T,x)\nabla c_i\cdot \nabla c_idx=\int_{\Omega}D_i(c,T,x)\nabla c_i\cdot \nabla (c_i-\gamma_i)dx=0,
			\end{gather}
			where we used that the $\gamma_i$ are constant, that $c_i-\gamma_i\in H^1_0(\Omega)$ and that $\operatorname{div}(D_i(c,T,x)\nabla c_i)=0$ in the weak sense. Further, $0<\lambda\leq D_i$ for all $1\leq i\leq M$ from which we conclude that $\|\nabla c_i\|_{L^2(\Omega)}=0$ and consequently $c_i=\gamma_i$ in $\Omega$ for all $1\leq i\leq M$ and hence $c=\gamma$.
			
			Further, we have
            \begin{gather}
                \nonumber
                \operatorname{div}(\epsilon \nabla (\phi(c,T,x)))=q\cdot c=q\cdot \gamma\quad \text{ and }\quad\phi(c,T,x)|_{\partial\Omega}=\phi_0,
            \end{gather}
            so that $\phi(c,T,x)=\phi(\gamma,T,x)$ coincides with the unique weak solution to (\ref{S3E12}) with $v=q\cdot\gamma$. By uniqueness, we find that $\phi(\gamma,T,x)=L^{-1}_{\epsilon}(q\cdot \gamma)$, and consequently
            \[                T=h(\gamma,\phi(\gamma,T,x),x)=h(\gamma,L^{-1}_{\epsilon}(q\cdot \gamma)(x),x)=\theta.
            \]
            So overall $(c,T)=(\gamma,\theta)$, and the solution is unique.
		\end{proof}
		\begin{rem}
			\label{S3R6}
			There is no uniqueness in the general setting of \Cref{S3L1}, as can already be seen in the scalar setting $N=1$. Suppose that $\partial\Omega\in C^2$. Let $u\in H^1_0(\Omega)$ be an eigenfunction of the Dirichlet Laplacian in $\Omega$, with eigenvalue $\lambda>0$. By elliptic regularity, $u\in H^2(\Omega)\subset C^0(\overline{\Omega})$. We can then let $\rho\colon\mathbb{R}\rightarrow \mathbb{R}$ be a smooth bump function which is compactly supported and identical $1$ on the image of  $u$. We can then set $\nu_1\equiv 1$ and $G_1(\eta,x):=\lambda\eta \rho(\eta)$. Then $G_1$ is globally bounded and the boundary value problem
			\begin{gather}
				\nonumber
				-\Delta \eta=G_1(\eta,x)=\lambda\eta(x) \rho(\eta(x)), \qquad \eta|_{\partial\Omega}=0
			\end{gather}
			admits the solutions $\eta\equiv 0$ and $\eta=u$. Thus, in the presence of sources additional constraints need to be imposed to guarantee uniqueness. The same example extends to the setting $N\geq 2$ by selecting $\nu_1\equiv 1$, $G_1(\eta,x):=\lambda\eta_1 \rho(\eta_1)$ and $\nu_i\equiv 1$, $G_i\equiv 0$ for $2\leq i\leq N$ and  considering $\eta|_{\partial\Omega}=0$ with the two distinct solutions $\eta=0$ and $\eta=(u,0,\dots,0)$.
		\end{rem}
		\subsection{Boundary measurements}
		\subsubsection{Voltage boundary measurements}
		We prove here that voltage boundary measurements $\phi(\gamma(x),\tau(x),x)-\phi(\gamma(y),\tau(y),y)$ uniquely determine $\phi(p,s,x)$ up to a constant along the boundary $\partial\Omega$. This is a key ingredient of the proof of \Cref{S2T1}.
		\begin{lem}
			\label{S3L7}
			Let $\Omega\subset\mathbb{R}^3$ be a bounded domain with connected $C^1$-boundary, $M\in \mathbb{N}$ and $\phi^1,\phi^2\in C^1(\mathbb{R}^{M+1}\times \overline{\Omega})$. The following two statements are equivalent:
			\begin{enumerate}
				\item For every $(\gamma,\tau)\in \bigl(W^{\frac{1}{2},2}(\partial\Omega)\bigr)^{M+1}$ we have $$\phi^1(\gamma(x),\tau(x),x)-\phi^1(\gamma(y),\tau(y),y)=\phi^2(\gamma(x),\tau(x),x)-\phi^2(\gamma(y),\tau(y),y)\quad\text{for $\mathcal{H}^2$-a.e. $x,y\in \partial\Omega$;}$$
				\item There exists some $r\in \mathbb{R}$ such that for all $(p,s,x)\in \mathbb{R}^{M+1}\times \partial\Omega$ we have $\phi^1(p,s,x)-\phi^2(p,s,x)=r$.
			\end{enumerate}
		\end{lem}
		\begin{proof}[Proof of \Cref{S3L7}]
			The implication $(ii)\Rightarrow (i)$ is immediate. For the converse, we notice that the temperature $\tau$ does not play any distinguished role in the context of \Cref{S3L7} so that we set
            \begin{gather}
                \nonumber
             N:=M+1\text{, }\rho:=(\gamma,\tau)\text{ and } z:=(p,s).
            \end{gather}
            We need to show that the identity
            \begin{gather}
                \nonumber
                \phi^1(\rho(x),x)-\phi^1(\rho(y),y)=\phi^2(\rho(x),x)-\phi^2(\rho(y),y)\text{ for all }\rho\in \left(W^{\frac{1}{2},2}(\partial\Omega)\right)^N\text{ and }\mathcal{H}^2\text{-a.e. } x,y\in \partial\Omega
            \end{gather}
            implies that $\phi^1(z,x)-\phi^2(z,x)=r$ for all $(z,x)\in \mathbb{R}^N\times \partial\Omega$ and some  $r\in\R$ independent of $(z,x)$.
            
            We notice that if $\rho\in (C^1(\partial\Omega))^N$, then $\phi^1(\rho(\cdot),\cdot)\in C^1(\partial\Omega)$, and the identity
            \begin{gather}
                \nonumber
                \phi^1(\rho(x),x)-\phi^1(\rho(y),y)=\phi^2(\rho(x),x)-\phi^2(\rho(y),y),\qquad x,y\in\partial\Omega,
            \end{gather}
             implies that $\nabla^{\partial\Omega}_x(\phi^1(\rho(x),x))=\nabla^{\partial\Omega}_x(\phi^2(\rho(x),x))$ for all $x\in \partial\Omega$, where $\nabla^{\partial\Omega}_x$ denotes the tangential gradient. This yields 
			\begin{gather}
				\label{S3E13}
				\sum_{i=1}^N(\partial_{z_i}\phi^1)(\rho(x),x)\nabla^{\partial\Omega}_x\rho_i(x)+(\nabla^{\partial\Omega}_x\phi^1)(\rho(x),x)=\sum_{i=1}^N(\partial_{z_i}\phi^2)(\rho(x),x)\nabla^{\partial\Omega}_x\rho_i(x)+(\nabla^{\partial\Omega}_x\phi^2)(\rho(x),x).
			\end{gather}
			 We now fix $z\in \mathbb{R}^N$ and consider $\rho(x)\equiv z \in (C^1(\partial\Omega))^N$. We then find $(\nabla^{\partial\Omega}_x\rho_i)\equiv 0$ since each $\rho_i$ is constant, and so (\ref{S3E13}) implies
			\begin{gather}
				\label{S3E14}
				(\nabla^{\partial\Omega}_x\phi^1)(z,x)=(\nabla^{\partial\Omega}_x\phi^2)(z,x),\qquad (z,x)\in \mathbb{R}^N\times \partial\Omega.
			\end{gather}
			Combining (\ref{S3E14}) with (\ref{S3E13}) then yields the identity
			\begin{gather}
				\label{S3E15}
				\sum_{i=1}^N(\partial_{z_i}\phi^1)(\rho(x),x)\nabla^{\partial\Omega}_x\rho_i(x)=\sum_{i=1}^N(\partial_{z_i}\phi^2)(\rho(x),x)\nabla^{\partial\Omega}_x\rho_i(x),\qquad \rho\in (C^1(\partial\Omega))^N,\;x\in \partial\Omega.
			\end{gather}
			Now fix $z\in \mathbb{R}^N$, $x_0\in \partial\Omega$ and  $1\leq j\leq N$. Pick $\rho_i(x)\equiv z_i$ for $i\neq j$, which again implies $\nabla^{\partial\Omega}_x\rho_i(x)\equiv 0$ for $i\neq j$. Then pick $\rho_j\in C^1(\partial\Omega)$ such that $\rho_j(x_0)=z_j$ and  $(\nabla^{\partial\Omega}_x\rho_j)(x_0)\neq 0$ (which is clearly always possible by a local construction). Then (\ref{S3E15}) implies
            \begin{gather}
                \nonumber
                (\partial_{z_j}\phi^1)(z,x_0)(\nabla^{\partial\Omega}_x\rho_j)(x_0)=(\partial_{z_j}\phi^2)(z,x_0)(\nabla^{\partial\Omega}_x\rho_j)(x_0),
            \end{gather}
            and in turn we have $(\partial_{z_j}\phi^1)(z,x_0)=(\partial_{z_j}\phi^2)(z,x_0)$. Since $z$, $x_0$ and $j$ were arbitrary, we deduce
			\begin{gather}
				\label{S3E16}
				(\partial_{z_i}\phi^1)(z,x)=(\partial_{z_i}\phi^2)(z,x),\qquad (z,x)\in \mathbb{R}^N\times \partial\Omega\text{, }1\leq i\leq N.
			\end{gather}
			Now (\ref{S3E14}) and (\ref{S3E16}) imply that the map $\phi^1-\phi^2:\mathbb{R}^N\times \partial\Omega\rightarrow \mathbb{R}$, $(z,x)\mapsto \phi^1(z,x)-\phi^2(z,x)$ has a vanishing gradient. Since by assumption $\partial\Omega$ is connected, so is $\mathbb{R}^N\times \partial\Omega$ and hence we conclude that $\phi^1(z,x)-\phi^2(z,x)=r$ for some $r\in \mathbb{R}$  independent of $(z,x)$.
		\end{proof}
		\subsubsection{Proof of \texorpdfstring{\Cref{S2T1}}{Theorem S2T1}}
        \label{ProofBoundary}
		\begin{proof}[Proof of \Cref{S2T1}]
			According to \Cref{S3L7} conditions (i,a) and (ii,b) are equivalent so that we may assume throughout that $\phi^1(p,s,x)-\phi^2(p,s,x)=r$ for some $r\in \mathbb{R}$ and all $(p,s,x)\in \mathbb{R}^{M+1}\times \partial\Omega$.
            We further notice that $\Lambda_{\operatorname{DN}}[\phi]$ depends on $\phi$ only through its gradient, recall (\ref{S2E1}), so that we may suppose that $r=0$ throughout the upcoming considerations.
  
            We are left with showing that, under the above assumptions, conditions (i,b) and (ii,a) are equivalent.
            
            We notice that $\phi^1(p,s,x)=\phi^2(p,s,x)$ for all $(p,s,x)\in \mathbb{R}^{M+1}\times \partial\Omega$ implies $(\nabla^{\partial\Omega}_x\phi^1)(p,s,x)=(\nabla^{\partial\Omega}_x\phi^2)(p,s,x)$ for all $(p,s,x)\in \mathbb{R}^{M+1}\times \partial\Omega$ where $\nabla^{\partial\Omega}_x$ denotes the tangential gradient. We hence have to show that (i,b) is equivalent to the condition $\mathcal{N}\cdot (\nabla_x \phi^1)(p,s,x)=\mathcal{N}\cdot (\nabla_x \phi^2)(p,s,x)$ for all $(p,s,x)\in \mathbb{R}^{M+1}\times \partial\Omega$.

            To this end,  for fixed $(\gamma,\tau)\in \bigl( W^{\frac{1}{2},2}(\partial\Omega)\bigr)^{M+1}$, set
            \begin{gather}
                \nonumber
                \eta_0(x):=\phi^1(\gamma(x),\tau(x),x).
            \end{gather}
            For given $v\in L^2(\Omega)$, denote by $L^{-1}_{\epsilon}(v)$ the unique  ($H^1(\Omega)$-)weak solution of the boundary value problem
			\begin{gather}
				\label{S3E17}
				\operatorname{div}(\epsilon \nabla L^{-1}_{\epsilon}(v))=v,\quad L^{-1}_{\epsilon}(v)|_{\partial\Omega}=\eta_0.
			\end{gather}
			We now let $(c,T)$ and $(\overline{c},\overline{T})$ denote the solutions to (\ref{S2E1}) with boundary data $(\gamma,\tau)$ and with respect to the potentials $\phi^1$ and $\phi^2$, respectively. We set
            \begin{gather}
                \nonumber
\sigma(x):=\phi^1(c(x),T(x),x)\text{ and }\overline{\sigma}(x):=\phi^2(\overline{c}(x),\overline{T}(x),x).
            \end{gather}
            On the one hand, according to (\ref{S2E1}), we have
            \begin{gather}
                \nonumber
                \operatorname{div}(\epsilon \nabla\sigma)=c\cdot q\text{ and }\operatorname{div}(\epsilon \nabla \overline{\sigma})=\overline{c}\cdot q.
            \end{gather}
            On the other hand, since $\phi^1$ and $\phi^2$ coincide on the boundary,
            \begin{gather}
                \nonumber
                \sigma|_{\partial\Omega}=\phi^1(\gamma(x),\tau(x),x)=\eta_0(x)=\phi^2(\gamma(x),\tau(x),x)=\overline{\sigma}|_{\partial\Omega}.
            \end{gather}
            We conclude that $\sigma=L^{-1}_{\epsilon}(c\cdot q)$ and $\overline{\sigma}=L^{-1}_{\epsilon}(\overline{c}\cdot q)$.
			Since by assumption the $D_i$ and $g_i$ depend on position alone, we see that $c_i$ and $\overline{c}_i$ satisfy the same linear elliptic PDE, (\ref{S2E1}), with the same boundary conditions. Thus, $c(x)=\overline{c}(x)$ throughout $\Omega$.
            
            This implies that the Dirichlet-to-Neumann maps $\Lambda_{\operatorname{DN}}[\phi^1]$ and $\Lambda_{\operatorname{DN}}[\phi^2]$ coincide if and only if $\mathcal{N}\cdot \nabla T=\mathcal{N}\cdot \nabla \overline{T}$ for all prescribed Dirichlet data. Further,
            \begin{gather}
                \nonumber
                \sigma=L^{-1}_{\epsilon}(c\cdot q)=L^{-1}_{\epsilon}(\overline{c}\cdot q)=\overline{\sigma}.
            \end{gather}
			We can apply the gradient to this identity, and we obtain for $x\in\Omega$:
			\begin{gather}
				\nonumber
				(\partial_s\phi^1)(c(x),T(x),x)\nabla T(x)+\sum_{i=1}^M(\partial_{p_i}\phi^1)(c(x),T(x),x)\nabla c_i(x)+(\nabla_x\phi^1)(c(x),T(x),x)
				\\
				\nonumber
				=(\partial_s\phi^2)(\overline{c}(x),\overline{T}(x),x)\nabla \overline{T}(x)+\sum_{i=1}^M(\partial_{p_i}\phi^2)(\overline{c}(x),\overline{T}(x),x)\nabla \overline{c}_i(x)+(\nabla_x\phi^2)(\overline{c}(x),\overline{T}(x),x).
			\end{gather}
			Now, since $\phi^1(p,s,x)=\phi^2(p,s,x)$ for all $(p,s,x)\in \mathbb{R}^{M+1}\times \partial\Omega$, we find $(\partial_s\phi^1)(p,s,x)=(\partial_s\phi^2)(p,s,x)$  and  $(\partial_{p_i}\phi^1)(p,s,x)=(\partial_{p_i}\phi^2)(p,s,x)$ for all $(p,s,x)\in \mathbb{R}^{M+1}\times \partial\Omega$.

			Since $c(x)=\overline{c}(x)$ in $\Omega$, the identities $\mathcal{N}\cdot (D_i(x)\nabla c_i(x))=\mathcal{N}\cdot (D_i(x)\nabla \overline{c}_i)$ are always satisfied. Recalling that $\partial_s\phi^1$ and $\partial_s\phi^2$, as well as $T$ and $\overline{T}$, coincide on the boundary, we have for $x\in\partial\Omega$
			\begin{gather}
				\nonumber
				(\partial_s\phi^1)(\gamma(x),\tau(x),x)\mathcal{N}\cdot \nabla T(x)+\mathcal{N}(x)\cdot(\nabla_x\phi^1)(\gamma(x),\tau(x),x)
				\\
\nonumber
				=(\partial_s\phi^1)(\gamma(x),\tau(x),x)\mathcal{N}\cdot \nabla \overline{T}(x)+\mathcal{N}(x)\cdot(\nabla_x\phi^2)(\gamma(x),\tau(x),x).
			\end{gather}
			We deduce  that the condition $\mathcal{N}\cdot \nabla T=\mathcal{N}\cdot \nabla \overline{T}$ is equivalent to $\mathcal{N}\cdot (\nabla_x\phi^1)(\gamma(x),\tau(x),x)=\mathcal{N}\cdot (\nabla_x\phi^2)(\gamma(x),\tau(x),x)$ for all $(\gamma,\tau)\in \bigl(W^{\frac{1}{2},2}(\partial\Omega)\bigr)^{M+1}$ and all $x\in \partial\Omega$, which in turn is equivalent to the identity $\mathcal{N}\cdot (\nabla_x\phi^1)(p,s,x)=\mathcal{N}\cdot (\nabla_x\phi^2)(p,s,x)$ for all $(p,s,x)\in \mathbb{R}^{M+1}\times \partial\Omega$.
			 This concludes the proof.
		\end{proof}
		\subsection{Interior measurements}
        \label{InteriorMeasure}
		\subsubsection{A Linearisation lemma}
		Before we come to the proof of \Cref{S2T3}, we will prove a linearisation lemma, which may be regarded as a generalisation of a result by Sun \cite[Lemma 2.10]{Sun96}.
		\begin{lem}[Linearisation lemma]
			\label{S3L8}
			Let $\Omega\subset\mathbb{R}^3$ be a bounded domain with (possibly disconnected) $C^{1,\beta}$-boundary for some $0<\beta\leq 1$. Let $M,N\in \mathbb{N}$, $0<\alpha\leq 1$, $\nu_*>0$ and suppose that $\nu=(\nu_1,\dots,\nu_M)\in C^{0,\alpha}_b\left(\mathbb{R}^N\times \overline{\Omega},\mathbb{R}^M\right)$ with $\nu_*\leq \nu_i$ for all $1\leq i\leq M$. Further, let $A\colon\left(L^2(\Omega)\right)^M\rightarrow \left(L^2(\Omega)\right)^N$ be a continuous function with the following two properties:
			\begin{enumerate}
				\item For every $\mu\in \mathbb{R}^M\subset \left(L^2(\Omega)\right)^M$ there is some $0<\alpha_{\mu}\leq 1$ such that $A(\mu)\in \left(C^{0,\alpha_{\mu}}(\overline{\Omega})\right)^M$,
				\item For every $\mu\in \mathbb{R}^M$ there is some $L_{\mu}\in (0,\infty)$ such that $\|A(\mu+v)-A(\mu)\|_{L^2(\Omega)}\leq L_{\mu}\|v\|_{L^2(\Omega)}$ for all $v\in \left(L^2(\Omega)\right)^M$.
			\end{enumerate}
			Take $f\in \left(C^{1,\beta}(\partial\Omega)\right)^M$ and $\mu\in \mathbb{R}^M$. For $t>0$, let $F^t\in \left(H^1(\Omega)\right)^M$ and $u\in \left(H^1(\Omega)\right)^M$ be weak solutions to the systems of equations
			\begin{gather}
				\label{S3E19}
				\operatorname{div}\left(\nu_i(A(F^t)(x),x)\nabla F^t_i(x)\right)=0\text{, }1\leq i\leq M\text{, }F^t|_{\partial\Omega}=\mu+tf,
				\\
				\label{S3E20}
				\operatorname{div}(\nu_i(A(\mu)(x),x)\nabla u_i(x))=0\text{, }1\leq i\leq M\text{, }u|_{\partial\Omega}=f.
			\end{gather}
			Then, for every $1\leq i\leq M$,
			\begin{gather}
				\label{S3E21}
				\frac{\mathcal{N}\cdot \left( \nu_i(A(F^t)(x),x)\nabla F^t_i\right)}{t}\rightarrow \mathcal{N}\cdot \left( \nu_i(A(\mu)(x),x)\nabla u_i\right)\text{ as }t\rightarrow 0\text{ in }W^{-\frac{1}{2},2}(\partial\Omega).
			\end{gather}
		\end{lem}
		\begin{rem}
			\label{S3R9}
			The elliptic equation (\ref{S3E20}) is a standard linear elliptic PDE for which uniqueness and existence results are available. The existence and uniqueness questions regarding (\ref{S3E19}) are more involved and we only point out that under the assumptions of \Cref{S3L8} the existence of solutions can be proven in a similar fashion as in \Cref{S3L1}. The statement of \Cref{S3L8} is that no matter what solution $F^t$ we select for the system (\ref{S3E19}), we will always observe the convergence (\ref{S3E21}) in the limit $t\rightarrow 0$.
		\end{rem}
		\begin{proof}[Proof of \Cref{S3L8}]
			\textit{Strategy: }In view of (\ref{S3E21}), we need to establish the convergence of the normal traces of certain vector fields. Instead of working directly with the normal traces, we will show that the corresponding vector fields converge in $H(\Omega,\operatorname{div})$. Due to the continuity of the normal trace, \cite[I \S 2 Theorem 2.5]{GR86}, this will imply the convergence of the normal traces.
            
            \textit{Step 1: Boundedness of $\kappa^t:=F^t-\mu$ in $H^1$.} Observe that $\kappa^t_i$ solves the following boundary value problem
			\begin{gather}
				\label{S3E22}
				\operatorname{div}\left(\tilde\nu_i\nabla \kappa^t_i\right)=0,\quad \kappa^t_i|_{\partial\Omega}=tf_i,
			\end{gather}
            where
            \begin{gather}
                \nonumber
                \tilde{\nu}_i(x):=\nu_i(A(\mu+\kappa^t)(x),x).
            \end{gather}
			According to \cite[Theorem 8.34]{GT01} and our assumptions on $A$ and $\nu$ there is a unique solution $u=(u_1,\dots,u_M)\in \left(C^{1}(\overline{\Omega})\right)^M$ to (\ref{S3E20}). We observe also that $\kappa^t_i-tu_i|_{\partial\Omega}=0$. 
            This allows us to compute
			\begin{gather}
				\label{eq:kappakappa}
				\int_{\Omega} \tilde{\nu}_i\nabla \kappa^t_i\cdot \nabla \kappa^t_idx=t\int_{\Omega}\tilde{\nu}_i\nabla \kappa^t_i\cdot \nabla u_idx+\int_{\Omega} \tilde{\nu}_i\nabla \kappa^t_i\cdot \nabla (\kappa^t_i-tu_i)dx=  t\int_{\Omega}\tilde{\nu}_i\nabla \kappa^t_i\cdot \nabla u_idx
			\end{gather}
			where we integrated by parts and used (\ref{S3E22}) in the last step. Since the $\nu_i$ are uniformly bounded below and above, we  deduce that
			\begin{gather}
				\label{S3E23}
				\|\nabla \kappa^t_i\|_{L^2(\Omega)}\leq \lambda |t|\|\nabla u_i\|_{L^2(\Omega)},
			\end{gather}
			where $\lambda>0$ denotes a generic constant which may change its value from line to line and may depend on $\nu$ and $\Omega$, but is always independent of $t$, $\kappa^t$, $f$, $\mu$ and $A$. By means of Poincar\'{e}'s inequality we find
			\begin{gather}
				\nonumber
				\|\kappa^t_i\|_{L^2(\Omega)}\leq |t|\|u_i\|_{L^2(\Omega)}+\|\kappa^t_i-tu_i\|_{L^2(\Omega)}\leq |t|\|u_i\|_{L^2(\Omega)}+\lambda \|\nabla \kappa^t_i-t\nabla u_i\|_{L^2(\Omega)}\leq \lambda |t|\|u_i\|_{H^1(\Omega)}, 
			\end{gather}
			where we used the triangle inequality and (\ref{S3E23}) in the last step. We hence conclude that
			\begin{gather}
				\label{S3E24}
				\|\kappa^t_i\|_{H^1(\Omega)}\leq \lambda |t|\|u_i\|_{H^1(\Omega)}.
			\end{gather}
			\textit{Step 2: $L^2(\Omega)$-convergence of the vector fields.} 
            Set $\nu_i^\mu(x):=\nu_i(A(\mu)(x),x)$. By \eqref{eq:kappakappa}, we have
			\begin{gather}
				\nonumber
				\int_{\Omega}\tilde{\nu}_i|\nabla \kappa^t_i-t\nabla u_i|^2dx=-t\int_{\Omega}\tilde{\nu}_i\nabla \kappa^t_i\cdot \nabla u_idx+t^2\int_{\Omega} \tilde{\nu}_i|\nabla u_i|^2dx
				\\
				\label{S3E25}
				=-t\int_{\Omega}\tilde{\nu}_i(\nabla \kappa^t_i-t\nabla u_i)\cdot \nabla u_idx=t\int_{\Omega}(\nu^\mu_i-\tilde{\nu}_i)(\nabla \kappa^t_i-t\nabla u_i)\cdot \nabla u_idx
			\end{gather}
			where we used (\ref{S3E20}) in the last step. By the H\"{o}lder continuity of $\nu$, we have
			\begin{gather}
				\nonumber
				|\tilde{\nu}_i-\nu^\mu_i|=|\nu_i(A(\mu+\kappa^t)(x),x)-\nu_i(A(\mu)(x),x)|\leq \lambda |A(\mu+\kappa^t)(x)-A(\mu)(x)|^{\alpha}.
			\end{gather}
			We combine this with (\ref{S3E25}) and obtain
			\begin{gather}
				\nonumber
				\int_{\Omega}\tilde{\nu}_i|\nabla \kappa^t_i-t\nabla u_i|^2dx\leq \lambda |t|\|\nabla \kappa^t_i-t\nabla u_i\|_{L^2(\Omega)}\||A(\mu+\kappa^t)-A(\mu)|^{\alpha}|\nabla u_i|\|_{L^2(\Omega)}
				\\
				\label{S3E26}
				\leq\lambda |t|\|\nabla \kappa^t_i-t\nabla u_i\|_{L^2(\Omega)}\|A(\mu+\kappa^t)-A(\mu)\|^{\alpha}_{L^2(\Omega)}\|\nabla u_i\|_{L^\infty(\Omega)}.
			\end{gather}
			Recall that $u_i\in C^1(\overline{\Omega})$ and thus $\|\nabla u_i\|_{L^\infty(\Omega)}<\infty$. By properties of $A$ and \eqref{S3E24} we find
            \begin{gather}
                \nonumber
                \|A(\mu+\kappa^t)-A(\mu)\|_{L^2(\Omega)}\leq L_{\mu}\|\kappa^t\|_{L^2(\Omega)}\le \lambda L_\mu |t|\|u\|_{H^1(\Omega)}.
            \end{gather}
            Then, by (\ref{S3E26}) we have
			\begin{gather}
				\nonumber
				\int_{\Omega}\tilde{\nu}_i|\nabla \kappa^t_i-t\nabla u_i|^2dx\leq \lambda L^{\alpha}_{\mu}\|u\|^{\alpha}_{H^1(\Omega)}|t|^{1+\alpha}\|\nabla u_i\|_{L^\infty(\Omega)}\|\nabla \kappa^t_i-t\nabla u_i\|_{L^2(\Omega)}.
			\end{gather}
			Finally, using that $\tilde{\nu}_i$ is uniformly bounded from below, we obtain
			\begin{gather}
				\label{S3E27}
				\left\|\nabla \frac{\kappa^t_i}{t}-\nabla u_i\right\|_{L^2(\Omega)}\leq L^{\alpha}_{\mu}\lambda |t|^{\alpha}\|u\|^{\alpha}_{H^1(\Omega)}\|\nabla u_i\|_{L^\infty(\Omega)}\rightarrow 0\text{ as }t\rightarrow 0.
			\end{gather}
            
			Since $\nabla F^t_i=\nabla \kappa^t_i$,  we can estimate
			\begin{gather}
				\nonumber
				\left\|\frac{\nu_i(A(F^t)(x),x)\nabla F^t_i}{t}-\nu_i(A(\mu)(x),x)\nabla u_i\right\|_{L^2(\Omega)}=\left\|\tilde{\nu}_i\frac{\nabla \kappa_i}{t}-\nu^\mu_i\nabla u_i\right\|_{L^2(\Omega)}
				\\
				\label{S3E28}
				\leq \left\|\tilde{\nu}_i\left(\frac{\nabla \kappa_i}{t}-\nabla u_i\right)\right\|_{L^2(\Omega)}+\|(\tilde{\nu}_i-\nu_i^\mu)\nabla u_i\|_{L^2(\Omega)}.
			\end{gather}
			Since $\tilde{\nu}_i$ is uniformly bounded above, by (\ref{S3E27}) the first summand on the right hand side of (\ref{S3E28}) converges to zero as $t\to0$. As for the second term, we can as before use the estimate
			\begin{gather}
				\nonumber
				\|(\tilde{\nu}_i-\nu^\mu_i)\nabla u_i\|_{L^2(\Omega)}\leq \lambda \||A(\mu+\kappa^t)-A(\mu)|^{\alpha}|\nabla u_i|\|_{L^2(\Omega)}\leq \lambda L^{\alpha}_{\mu}\|\nabla u_i\|_{L^\infty(\Omega)}\|u\|^{\alpha}_{H^1(\Omega)}|t|^{\alpha}
			\end{gather}
			which converges to zero as $t\rightarrow 0$. Combining this with (\ref{S3E28}) we find
			\begin{gather}
				\nonumber
				\left\|\frac{\nu_i(A(F^t)(x),x)\nabla F^t_i}{t}-\nu_i(A(\mu)(x),x)\nabla u_i\right\|_{L^2(\Omega)}\rightarrow 0\text{ as }t\rightarrow 0.
			\end{gather}
			\textit{Step 3: Concluding the proof.} Since $\operatorname{div}\left(\frac{\nu_i(A(F^t)(x),x)\nabla F^t_i}{t}-\nu_i(A(\mu)(x),x)\nabla u_i\right)=0$ in view of (\ref{S3E19})-(\ref{S3E20}), we infer from step 2 that
			\begin{gather}
				\nonumber
				\frac{\nu_i(A(F^t)(x),x)\nabla F^t_i}{t}\rightarrow\nu_i(A(\mu)(x),x)\nabla u_i\text{ in }H(\Omega,\operatorname{div})\text{ as }t\rightarrow 0.
			\end{gather}
			Thus, the lemma follows by the continuity of the normal trace operator, cf.\ \cite[I \S 2 Theorem 2.5]{GR86}.
		\end{proof}
		\subsubsection{Proof of \texorpdfstring{\Cref{S2T3}}{Theorem S2T3}}
		\begin{proof}[Proof of \Cref{S2T3}]
			
			\underline{(i)$\Rightarrow$ (ii):} \textit{Step 1: Reduction to a linear problem.} It follows first from \Cref{S3L7} and (i,a) that we have
            \begin{gather}
                \nonumber
                \phi^1(p,s,x)-\phi^2(p,s,x)=r\text{ for all }(p,s,x)\in \mathbb{R}^{M+1}\times \partial\Omega
            \end{gather}
            and some $r$ independent of $(p,s,x)$. Upon a normalisation we may without loss of generality assume that $r=0$.
            
            For $(\gamma,\tau)\in \bigl(W^{\frac{1}{2},2}(\partial\Omega)\bigr)^{M+1}$, we then let $(T^1,c^1)$ and $(T^2,c^2)$ denote any fixed solutions to (\ref{S2E3}) with $c^1|_{\partial\Omega}=c^2|_{\partial\Omega}=\gamma$ and $T^1|_{\partial\Omega}=T^2|_{\partial\Omega}=\tau$, with coefficients $(D^1_i,\phi^1,\epsilon,q)$ and $(D^2_i,\phi^2,\epsilon,q)$ respectively, cf.\ \Cref{S3C4}. Further define
            \begin{equation}\label{eq:sigmatildesigma}
\sigma^1(x):=\phi^1(c^1(x),T^1(x),x) \text{ and }\sigma^2(x):=\phi^2(c^2(x),T^2(x),x).
            \end{equation}
            As in the proof of \Cref{S2T1}, we set $\eta_0(x):=\phi^1(\gamma(x),\tau(x),x)$ and deduce from (\ref{S2E3}) that
            \begin{equation}\label{eq:sigmatildesigma2}
                \sigma^1=L^{-1}_{\epsilon}(c^1 \cdot q)\text{, }\sigma^2=L^{-1}_{\epsilon}(c^2 \cdot q),
            \end{equation}
            where for given $v\in L^2(\Omega)$, $L^{-1}_{\epsilon}(v)$ denotes the unique weak solution of the boundary value problem
			\begin{gather}
				\label{S3E29}
				\operatorname{div}(\epsilon \nabla L^{-1}_{\epsilon}(v))=v \text{ and }L^{-1}_{\epsilon}(v)|_{\partial\Omega}=\eta_0.
			\end{gather}
			According to \Cref{S3L3}, $\phi^1$ and $\phi^2$ admit inverse functions $h^1,h^2\in \dot{C}^1_b(\mathbb{R}^{M+1}\times \overline{\Omega})$, respectively, so that \eqref{eq:sigmatildesigma} implies
            \begin{gather}
                \nonumber
                T^1(x)=h^1(c^1(x),\sigma^1(x),x)\text{ and }T^2(x)=h^2(c^2(x),\sigma^2(x),x).
            \end{gather}
            By \eqref{eq:sigmatildesigma2} we obtain
			\begin{gather}
				\label{S3E30}
				T^1(x)=h^1(c^1(x),L^{-1}_{\epsilon}(c^1\cdot q)(x),x)\text{ and }T^2(x)=h^2(c^2,L^{-1}_{\epsilon}(c^2\cdot q)(x),x).
			\end{gather}
			We now define
            \begin{gather}
                \label{Nu}
    \nu^1_i(p,s,x):=D^1_i(p,h^1(p,s,x),x)\text{ and }\nu^2_i(p,s,x):=D^2_i(p,h^2(p,s,x),x)
            \end{gather}
            and notice that $\nu^1_i,\nu^2_i\in C^{0,1}_b(\mathbb{R}^{M+1}\times \overline{\Omega})$. Further, it follows from (\ref{S2E3}) and (\ref{S3E30}) that for all $1\leq i\leq M$ we have
			\begin{gather}
				\label{S3E31}
				\operatorname{div}(\nu^1_i(c^1(x),L^{-1}_{\epsilon}(c^1\cdot q)(x),x)\nabla c^1_i)=0\text{, }c^1|_{\partial\Omega}=\gamma,
				\\
				\label{S3E32}
				\operatorname{div}(\nu^2_i(c^2(x),L^{-1}_{\epsilon}(c^2\cdot q)(x),x)\nabla c^2_i)=0\text{, }c^2|_{\partial\Omega}=\gamma.
			\end{gather}
            
			We now observe that, since $\phi^1(p,s,x)=\phi^2(p,s,x)$ for all $(p,s,x)\in \mathbb{R}^{M+1}\times \partial\Omega$, the same is true for their inverses, i.e.\ $h^1(p,s,x)=h^2(p,s,x)$ for all $(p,s,x)\in \mathbb{R}^{M+1}\times \partial\Omega$. This implies that for any fixed $\eta_0\in W^{\frac{1}{2},2}(\partial\Omega)$ and for any $\gamma\in \left(W^{\frac{1}{2},2}(\partial\Omega)\right)^M$ we may pick
            \begin{gather}
                \nonumber
                \tau(x):=h^1(\gamma(x),\eta_0(x),x)=h^2(\gamma(x),\eta_0(x),x),
            \end{gather}
            so that
            \begin{gather}
                \nonumber
                \phi^1(\gamma(x),\tau(x),x)=\eta_0(x)=\phi^2(\gamma(x),\tau(x),x).
            \end{gather}
            We then define, for fixed $\eta_0\in C^{1,\beta}(\partial\Omega)$, the following operator
			\begin{gather}
				\label{S3E33}
				A\colon\left(L^2(\Omega)\right)^M\rightarrow \left(L^2(\Omega)\right)^{M+1}\text{, }v\mapsto \left(v,L^{-1}_{\epsilon}(v\cdot q)\right)
			\end{gather}
			and observe that (\ref{S3E31}) and (\ref{S3E32}) become
			\begin{gather}
				\label{S3E34}
				\operatorname{div}\left(\nu^1_i(A(c^1)(x),x)\nabla c^1_i\right)=0\text{, }c^1|_{\partial\Omega}=\gamma\text{ and }\operatorname{div}\left(\nu^2_i(A(c^2)(x),x)\nabla c^2_i\right)=0\text{, }c^2|_{\partial\Omega}=\gamma. 
			\end{gather}
            
			Our goal now is to use \Cref{S3L8}. First, we recall that we have $\nu^1_i,\nu^2_i\in C^{0,1}_b(\mathbb{R}^{M+1}\times \overline{\Omega})$ and that $\nu^1_i$ and $\nu^2_i$ are bounded away from zero because so are $D^1_i$ and $D^2_i$. So we are left with verifying the properties of the operator $A$. Clearly $A$ maps $(L^2(\Omega))^M$ into $(L^2(\Omega))^{M+1}$. Further, if $\mu\in \mathbb{R}^M\subset (L^2(\Omega))^M$ is given and $\eta_0\in C^{1,\beta}(\partial\Omega)$, then according to \cite[Theorem 8.34]{GT01}
            \begin{gather}
                \nonumber
                L^{-1}_{\epsilon}(\mu\cdot q)\in C^{1,\beta}(\overline{\Omega})\text{ so that }A(\mu)\in \left(C^{0,\beta}(\overline{\Omega})\right)^{M+1}\text{ for every }\mu\in \mathbb{R}^M.
            \end{gather}
            We show now that $A$ is (globally) Lipschitz continuous. It is clearly enough to estimate
            \begin{gather}
                \nonumber
                \|L^{-1}_{\epsilon}((v+w)\cdot q)-L^{-1}_{\epsilon}(v\cdot q)\|_{L^2(\Omega)}\text{ for every }v,w\in L^2(\Omega).
            \end{gather}
            We observe that $\alpha:=L^{-1}_{\epsilon}((v+w)\cdot q)-L^{-1}_{\epsilon}(v\cdot q)$ is a weak solution to
            \begin{gather}
                \nonumber
                \operatorname{div}(\epsilon \nabla \alpha)=w\cdot q\text{ and }\alpha|_{\partial\Omega}=0.
            \end{gather}
            Thus, by classical energy estimates for elliptic PDEs, we obtain
            \(
            \|\alpha\|_{L^2(\Omega)}\le C(\Omega,\epsilon)|q|\|w\|_{L^2(\Omega)}.
            \)
            We deduce that $A$ is globally Lipschitz continuous.
			
			We conclude that for any fixed $\eta_0\in C^{1,\beta}(\partial\Omega)$, the operator $A$ and coefficients $\nu^1_i,\nu^2_i$ satisfy the requirements of \Cref{S3L8}. We follow here now the reasoning of \cite{Sun96} and observe that condition (i,b) of \Cref{S2T3} implies that
			for every fixed $f=(f_1,\dots,f_M)\in \left(C^{1,\beta}(\partial\Omega)\right)^M$, $\mu\in \mathbb{R}^M$ and $t>0$ there exist solutions $c^1_t$ and $c^2_t$ of (\ref{S3E34}) with
            \begin{gather}
                \nonumber
            \gamma=\mu+tf\text{ and }\mathcal{N}\cdot (\nu^1_i(A(c^1_t),x)\nabla c^1_{t,i})=\mathcal{N}\cdot (\nu^2_i(A(c^2_t),x)\nabla c^2_{t,i}).     
            \end{gather}
            It then follows from \Cref{S3L8}
			\begin{gather}
				\nonumber
				\mathcal{N}\cdot (\nu^1_i(A(\mu)(x),x)\nabla u^1_i)=\lim_{t\rightarrow 0}\frac{\mathcal{N}\cdot (\nu^1_i(A(c^1_t),x)\nabla c^1_{t,i})}{t}
				\\
				\label{S3E35}
				=\lim_{t\rightarrow 0}\frac{\mathcal{N}\cdot (\nu^2_i(A(c^2_t),x)\nabla c^2_{t,i})}{t}=\mathcal{N}\cdot (\nu^2_i(A(\mu)(x),x)\nabla u^2_i) 
			\end{gather}
			where $u^1,u^2\in (H^1(\Omega))^M$ weakly solve the linear PDEs
			\begin{gather}
				\label{S3E36}
				\operatorname{div}(\nu^1_i(A(\mu)(x),x)\nabla u^1_i)=0\text{ in }\Omega\text{, }u^1|_{\partial\Omega}=f\text{ and }\operatorname{div}(\nu^2_i(A(\mu)(x),x)\nabla u^2_i)\text{ in }\Omega\text{, }u^2|_{\partial\Omega}=f.
			\end{gather}
			We observe that (\ref{S3E36}) are decoupled systems of linear scalar elliptic PDEs which admit unique weak solutions. Thus, we can define single-valued Dirichlet-to-Neumann maps
            \begin{gather}
                \nonumber
                \Lambda_{\operatorname{DN}}[\nu^1_i]:W^{\frac{1}{2},2}(\partial\Omega)\rightarrow W^{-\frac{1}{2},2}(\partial\Omega)\text{, }\Lambda_{\operatorname{DN}}[\nu^2_i]:W^{\frac{1}{2},2}(\partial\Omega)\rightarrow W^{-\frac{1}{2},2}(\partial\Omega)
            \end{gather}
            which map $f_i\in W^{\frac{1}{2},2}(\partial\Omega)$ onto $\mathcal{N}\cdot (\nu^1_i(A(\mu)(x),x)\nabla u^1_i)$ and $\mathcal{N}\cdot (\nu^2_i(A(\mu)(x),x)\nabla u^2_i)$ respectively. Then (\ref{S3E35}), the density of $C^{1,\beta}(\partial\Omega)$ functions in $W^{\frac{1}{2},2}(\partial\Omega)$ \cite[Proposition 3.40]{DD12}, and the continuity of $\Lambda_{\operatorname{DN}}[\nu^1_i]$ and $\Lambda_{\operatorname{DN}}[\nu^2_i]$ imply
			\begin{gather}
				\label{S3E37}
				\Lambda_{\operatorname{DN}}[\nu^1_i]=\Lambda_{\operatorname{DN}}[\nu^2_i]\text{ for all }1\leq i\leq M.
			\end{gather}
			\textit{Step 2: Reconstruction of the coefficients.} We can now employ the standard linear, scalar Calder\'{o}n uniqueness result \cite[Theorem 1.1]{Hab15} and conclude
			\begin{gather}
				\label{S3E38}
				\nu^1_i(A(\mu)(x),x)=\nu^2_i(A(\mu)(x),x)\text{ for every }\mu\in \mathbb{R}^M\text{ and every }x\in \overline{\Omega},
			\end{gather}
			and  for any prescribed boundary condition $\eta_0\in C^{1,\beta}(\partial\Omega)$ for the solution operator $L^{-1}_{\epsilon}$ in (\ref{S3E29}).
			
			We can now fix any $\mu\in \mathbb{R}^M$, $s\in \mathbb{R}$ and let $\omega_0$  be the unique $C^{1,\beta}(\overline{\Omega})$ weak solution of
            \begin{gather}
                \nonumber
                \operatorname{div}(\epsilon \nabla \omega_0)=q\cdot \mu\text{ in }\Omega\text{ and }\omega_0|_{\partial\Omega}=0.
            \end{gather}
            We then fix any $y\in \overline{\Omega}$ and observe that $\omega(x):=\omega_0(x)+s-\omega_0(y)$ is the unique (weak) $C^{1,\beta}(\overline{\Omega})$-solution to
            \begin{gather}
                \nonumber
                \operatorname{div}(\epsilon \nabla \omega)=q\cdot \mu\text{ and }\omega|_{\partial\Omega}=s-\omega_0(y).
            \end{gather}
            Thus, for $\eta_0(x):=s-\omega_0(y)\in C^{1,\beta}(\partial\Omega)$ we find, recall (\ref{S3E29}), $\omega=L^{-1}_{\epsilon}(q\cdot\mu)$, so that
            \begin{gather}
\nonumber
L^{-1}_{\epsilon}(q\cdot \mu)(y)=\omega(y)=s.
            \end{gather}
             Consequently, for this choice of $\eta_0$, we obtain from (\ref{S3E38}), (\ref{S3E33}) and  (\ref{Nu})
			\begin{equation}\label{S3E39}
			\begin{split}D^1_i(\mu,h^1(\mu,s,y),y)&=D^1_i(\mu,h^1(\mu,L^{-1}_{\epsilon}(\mu\cdot q)(y),y),y)=\nu^1_i(\mu,L^{-1}_{\epsilon}(\mu\cdot q)(y),y)
				\\
&=\nu^2_i(\mu,L^{-1}_{\epsilon}(\mu\cdot q)(y),y)=D^2_i(\mu,h^2(\mu,s,y),y)
                \end{split}
			\end{equation}
			for every $(\mu,s,y)\in \mathbb{R}^{M+1}\times \overline{\Omega}$ and $1\leq i\leq M$. By \Cref{S3P5}, if we prescribe the constant boundary conditions
            \begin{gather}
                \nonumber
                c^1|_{\partial\Omega}=\mu=c^2|_{\partial\Omega},
            \end{gather}
            then it follows that $c^1(x)=\mu=c^2(x)$ throughout $\Omega$. We then make use of (\ref{S3E30}) and condition (i,c) of \Cref{S2T3} to conclude that, for the choice $\eta_0(x):=s-\omega_0(y)$,
			\begin{equation}\label{S3E40}
h^1(\mu,s,y)=h^1(c^1(y),L^{-1}_{\epsilon}(q\cdot c^1)(y),y)=T^1(y)
				=T^2(y)=h^2(c^2(y),L^{-1}_{\epsilon}(c^2\cdot q)(y),y)=h^2(\mu,s,y)
			\end{equation}
            for every $(\mu,s,y)\in \mathbb{R}^{M+1}\times \overline{\Omega}$.
			We deduce from this
            \begin{gather}
                \nonumber
\phi^1(p,s,x)=\phi^2(p,s,x)\text{ for all }(p,s,x)\in \mathbb{R}^{M+1}\times \overline{\Omega},
            \end{gather}
            since $h^1$ is the inverse of $\phi^1$, recall \Cref{S3L3}. Finally, for fixed $(p,t,x)\in \mathbb{R}^{M+1}\times \overline{\Omega}$ we can pick $s\in \mathbb{R}$ such that $t=h^1(p,s,x)=h^2(p,s,x)$. Plugging this in (\ref{S3E39}), we deduce
			\begin{gather}
			    \nonumber D^1_i(p,t,x)=D^2_i(p,t,x)\text{ for all }(p,t,x)\in \mathbb{R}^{M+1}\times \overline{\Omega}.
			\end{gather}
			This proves overall that $\phi^1=\phi^2$ and $D^1_i=D^2_i$ for all $1\leq i\leq M$. Hence we proved the implication (i) $\Rightarrow$ (ii).
			\newline
			\newline
			\underline{(ii) $\Rightarrow$ (i):} We observe that (\ref{S2E3}) only depends on the gradient of the potential $\phi$ so that, if $D^1_i=D^2_i$ for all $1\leq i\leq M$ and $\phi^1=\phi^2+r$ for some constant $r\in \mathbb{R}$, they will induce the same PDEs and hence the solution spaces coincide. So we only need to notice that the solution space is non-empty, \Cref{S3C4}, to see that (ii) implies (i).
		\end{proof}
		\section*{Acknowledgements}
		The research was supported in part by the MIUR Excellence Department Project awarded to Dipartimento di Matematica, Università di Genova, CUP D33C23001110001. This work was supported by the Italian Ministry of the Environment and Energy Security (MASE) through the Mission Innovation 2.0 project 'Sole di notte' (ID: MI\_ERE\_00192), CUP: F33C25001220001. Co-funded by the European Union (ERC, SAMPDE, 101041040). Views and opinions expressed are however those of the authors only and do not necessarily reflect those of the European Union or the European Research Council. Neither the European Union nor the granting authority can be held responsible for them. Co-funded by the European Union-Next Generation EU, Missione 4 Componente 1 CUP D53D23005770006. This project is funded by the European Union Horizon Europe Grant Agreement n.\ 101251004. The project is supported by the Clean Hydrogen Partnership and its members.

		\section*{Data availability}
		No new data was created in this study.
		\section*{Conflict of interest}
		The authors declare that they have no conflict of interest.
		\appendix
		\setcounter{equation}{0}\renewcommand\theequation{A\arabic{equation}}
		\section{The model}
		\label{Model}
		\subsection{Modelling the ion concentrations}
		\label{Concetrations}
		We recall that electrolysis is based on the following reactions (see (\ref{S1E1})-(\ref{S1E3}))
		\begin{gather}
			\nonumber
			\text{Anode: }2\text{OH}^{-}\rightarrow \text{H}_2\text{O}+\frac{\text{O}_2}{2}+2\text{e}^-
			\\
			\nonumber
			\text{Cathode: }2\text{H}_2\text{O}+2\text{e}^-\rightarrow \text{H}_2+2\text{OH}^{-}
			\\
			\nonumber
			\text{Overall reaction: }\text{H}_2\text{O}\rightarrow \text{H}_2+\frac{\text{O}_2}{2}
		\end{gather}
		In practice KOH may also be added in the case of an AEM electrolyser, but is in principle not required for an electrolyser to function \cite{Muh25}. We denote the concentration of the $i$-th particle species by $c_i$, where $c_i(t,x)$ is a function of time $t$ and position $x$. Let $N_i$ denote the $i$-th species particle flux and $g_i$ denote the source of the $i$-th particle. By mass conservation, also known as Fick's second law, it follows that
		\begin{gather}
			\label{S1E4}
			\partial_tc_i+\operatorname{div}(N_i)=g_i.
		\end{gather}
		Following Fick's first law we suppose that
        \begin{gather}
            \nonumber
            N_i=-D_i\nabla c_i,
        \end{gather}
        where $D_i$ is the diffusion coefficient of the $i$-th particle species. We further make the simplifying assumption that the diffusion tensor is scalar, i.e.\ the underlying medium is isotropic. Inserting this relationship into (\ref{S1E4}) gives us the following parabolic diffusion equation
		\begin{gather}
			\label{S1E5}
			\partial_tc_i-\operatorname{div}(D_i\nabla c_i)=g_i.
		\end{gather}
		We emphasise that (\ref{S1E5}) does not take into account electrostatic forces between the ions and that a more realistic model should incorporate such forces, since some of the particles within the device are charged and hence repel or attract each other. A more realistic model is presented in \Cref{Electrostatic}. However, its analysis is more challenging and left for future work; the present manuscript focuses on the simplified equations (\ref{S1E5}).
		
		In general, the diffusion coefficients $D_i$ and the source terms are unknown to us and may all depend on the temperature of the system and on the particle concentrations, i.e.
        \begin{gather}
            \nonumber
            D_i=D_i(T,c,x,t),\quad g_i=g_i(T,c,x,t),
        \end{gather}
        where $T$ is the (scalar) temperature and $c=(c_1,\dots,c_M)$, where $M$ is the number of distinct species. In turn, $c_i=c_i(t,x)$ and $T=T(t,x)$ are assumed to be functions of time and position.

		\subsection{Modelling the electric potential}
		\label{Potential}
		We start by considering the macroscopic Maxwell equations
		\begin{gather}
			\label{S1E6}
			\operatorname{div} D=\rho_f\text{, }\operatorname{curl} H=J_f+\partial_tD\text{, }\operatorname{div} B=0\text{, }\operatorname{curl} E=-\partial_tB,
		\end{gather}
		where $B$ is the magnetic field, $E$ is the electric field, $H$ is the magnetising field, $D$ is the electric displacement field, $\rho_f$ is the free charge density and $J_f$ is the free current density. The relationship between the distinct fields is given by
        \begin{gather}
            \nonumber
            H=\frac{1}{\mu_0}B-M\text{, }D=\epsilon_0E+P,
        \end{gather}
        where $P$ is the polarisation field, $M$ is the magnetisation field, and $\mu_0$ and $\epsilon_0$ are the vacuum permeability and permittivity, respectively. The $P$- and $M$-fields are in turn related to the bound charges and electric currents, but we make here the simplifying assumption that we are dealing with a linear medium so that we find the relationships
        \begin{gather}
            \nonumber
            M=\chi_m H\text{, }P=\chi_e\epsilon_0E,
        \end{gather}
        where $\chi_m$ and $\chi_e$ are the magnetic and electric susceptibility, cf.\ \cite[Chapter 4.4 \& 5.8]{Jack62}. We obtain
        \begin{gather}
            \nonumber
            B=\mu H\text{, }D=\epsilon E,
        \end{gather}
        where $\mu=\mu_0(1+\chi_m)$ is the magnetic permeability and $\epsilon=\epsilon_0(1+\chi_e)$ is the electric permeability. Consequently (\ref{S1E6}) becomes
		\begin{gather}
			\label{S1E7}
			\operatorname{div}(\epsilon E)=\rho_f\text{, }\operatorname{curl}\left(\frac{B}{\mu}\right)=J_f+\partial_t(\epsilon E)\text{, }\operatorname{div} B=0\text{, }\operatorname{curl}E=-\partial_tB.
		\end{gather}
		In practice, an electrolyser cell will occupy some contractible (or even convex) region $\Omega\subset\mathbb{R}^3$, so that the condition $\operatorname{div} B=0$ becomes equivalent to the statement $B=\operatorname{curl}A$ for some vector potential $A$. In turn
        \begin{gather}
            \nonumber
            0=\operatorname{curl}E+\partial_tB=\operatorname{curl}(E+\partial_tA).
        \end{gather}
        Since $\Omega$ is contractible this is equivalent to the statement $E=\nabla \phi-\partial_tA$, and so (\ref{S1E7}) becomes
		\begin{gather}
			\label{S1E8}
			\operatorname{div}(\epsilon E)=\rho_f\text{, }\operatorname{curl}\left(\frac{B}{\mu}\right)=J_f+\partial_t(\epsilon E)\text{, }B=\operatorname{curl}A\text{, }E=\nabla \phi-\partial_tA.
		\end{gather}
        
		We notice that the vector potential $A$ is unique up to the addition of gradient fields and that, in turn, $\phi$ depends on the choice of $A$. Our goal now will be to fix a specific gauge in order to be able to associate a physical meaning with the potential $\phi$. If $E$ is curl-free, or equivalently $B$ is static, then we may pick $A$ to be time-independent and obtain a scalar potential with $E=\nabla \phi$, which has a physical meaning. Such a choice cannot be made for time-dependent fields. In order to associate a physical meaning to $\phi$ in time-varying electromagnetic fields, we fix any two points $x,y\in \partial\Omega$. Using a voltmeter, we are able to assign the points $x,y$ a corresponding measurement. What a voltmeter actually measures is the quantity $\int_{\gamma_{\operatorname{vm}}}E$, which corresponds to the work, per charge, that is necessary to move a charge from point $x$ to point $y$ along the circuit path $\gamma_{\operatorname{vm}}$ within the voltmeter. In general time-dependent fields, this quantity will be path dependent and hence the reading of the voltmeter will depend on the placement of the voltmeter; cf.\ \cite{Romer82} for a discussion about voltmeter measurements in time-varying electromagnetic fields.
		
		However, if we make the simplifying assumption that the voltmeter is close to $\partial\Omega$, then $\gamma_{\operatorname{vm}}$ is well-approximated by a path $\gamma$ on $\partial\Omega$ that connects $x$ and $y$. Thus we find 
		\begin{gather}
			\nonumber
			\int_{\gamma_{\operatorname{vm}}}E\approx \int_{\gamma}E.
		\end{gather}
		Since $\gamma$ is tangent to $\partial\Omega$ we see that if we can find a vector potential $A$ such that $\partial_tA \perp \partial\Omega$, then
		\begin{gather}
			\nonumber
			\int_{\gamma_{\operatorname{vm}}}E\approx \int_{\gamma}E=\int_{\gamma}\nabla \phi=\phi(y)-\phi(x).
		\end{gather}
		Consequently, our voltmeter measurements correspond to potential differences on the boundary. Of course, if $E$ is a gradient field throughout $\mathbb{R}^3$, then this approximation becomes exact.
		\newline
		\newline
		Let us therefore examine when a potential $A$ can be chosen such that $\partial_tA\perp \partial\Omega$. Since $E=\nabla \phi-\partial_t A$, we see that we can find $\partial_tA\perp\partial\Omega$ if and only if $E^\parallel$ is a gradient field, where $E^\parallel$ denotes the part of $E|_{\partial\Omega}$ that is tangent to $\partial\Omega$. In our application, $\partial\Omega$ will be diffeomorphic to a sphere and so $E^{\parallel}$ is a gradient field if and only if $\operatorname{curl}_{\partial\Omega}(E^{\parallel})=0$, where $\operatorname{curl}_{\partial\Omega}$ denotes the curl operator on $\partial\Omega$. We now make use of the fact that
		\begin{gather}
			\nonumber
			\operatorname{curl}_{\partial\Omega}(E^{\parallel})=\mathcal{N}\cdot \operatorname{curl}E,
		\end{gather}
		where $\mathcal{N}$ denotes the outward unit normal and $\operatorname{curl}$ denotes the standard Euclidean curl. Since we have $E=\nabla \phi-\partial_tA$, we get $\operatorname{curl}E=-\partial_t\operatorname{curl}A=-\partial_tB$. Hence $E^\parallel$ is a gradient field if and only if $\mathcal{N}\cdot \partial_tB=0$ or equivalently if and only if $\mathcal{N}\cdot B$ is time independent. Therefore, as long as the normal part of the magnetic field $B$ does not change in time we are able to identify a potential $\phi$ with a physical significance.
		\newline
		\newline
		To uniquely fix the temporal gauge $\partial_tA$ we may prescribe its divergence as follows
		\begin{gather}
			\label{S1E9}
			\operatorname{curl}(\partial_tA)=\partial_tB\text{, }\operatorname{div}(\epsilon \partial_tA)=0\text{ in }\Omega\text{, }\partial_tA\perp \partial\Omega.
		\end{gather}
		It is easy to see that if we have two vector potentials $A_1,A_2$ satisfying (\ref{S1E9}), then their difference must be a gradient field $\nabla f$ that satisfies the elliptic equation
		\begin{gather}
			\nonumber
			\operatorname{div}(\epsilon \nabla \partial_tf)=0\text{ and } \nabla \partial_tf\perp \partial\Omega.
		\end{gather}
		We notice that the condition $\nabla \partial_tf\perp \partial\Omega$ means that $\partial_tf$ is constant on $\partial\Omega$ and, since we are only interested in the gradient of $f$, we may suppose that $\partial_tf|_{\partial\Omega}=0$. Therefore, if $\epsilon\in L^{\infty}(\Omega)$ is bounded away from zero, we find $\nabla \partial_tf=0$ and consequently $\partial_tA_1=\partial_tA_2$, as desired.
		
		The existence of a solution is guaranteed as long as $\mathcal{N}\cdot B$ is time-independent. As explained previously, this condition implies that $E^\parallel$ is a gradient field. Therefore, if $A$ is any fixed vector potential with corresponding scalar potential $\phi$ we find $\partial_tA^{\parallel}=\nabla_{\partial\Omega}\kappa$ for some suitable function $\kappa$. We can then let $f$ be a solution to the elliptic boundary value problem
		\begin{gather}
			\nonumber
			f|_{\partial\Omega}=-\kappa \text{ and } \operatorname{div}(\epsilon \nabla f)=-\operatorname{div}(\epsilon \partial_tA) \text{ in } \Omega.
		\end{gather}
		Then $\widetilde{A}:=A+\int_0^t\nabla f ds$ is a valid vector potential with
        \begin{gather}
            \nonumber
            \partial_t\widetilde{A}=\partial_tA+\nabla f\text{, i.e.\ }(\partial_t\widetilde{A})^{\parallel}=(\partial_tA)^{\parallel}-\nabla_{\partial\Omega}\kappa=0 \text{ and }\operatorname{div}(\epsilon\partial_t\widetilde{A})=0.
        \end{gather}
		Imposing the gauge (\ref{S1E9}) we obtain the following set of equations in terms of the electric potential $\phi$ and the magnetic vector potential $A$
		\begin{gather}
			\label{S1E10}
			\operatorname{div}(\epsilon \nabla \phi)=\rho_f\text{, }\operatorname{curl}\left(\frac{\operatorname{curl}(A)}{\mu}\right)=J_f+\partial_t(\epsilon (\nabla \phi-\partial_tA))\text{, }\operatorname{div}(\epsilon\partial_tA)=0\text{, }\partial_tA\perp\partial\Omega. 
		\end{gather}
		We notice that if $B$ is static, i.e.\ $\partial_tB=0$, then the unique solution to (\ref{S1E9}) is $\partial_tA=0$ and hence we recover the identity $E=\nabla \phi$ from classical electrostatics.
		
		We further point out that it is a priori not clear to what extent the condition $\partial_t(\mathcal{N}\cdot B)=0$ on $\partial\Omega$ is a good approximation of our system under consideration and that experiments should be conducted to understand this. In the present work, we investigate the static case $\partial_tB=0$ for which this condition is trivially satisfied.
		\newline
		\newline
		The quantity $\rho_f$ in (\ref{S1E10}) denotes the free charge density, and in our application the free charges are precisely the distinct ions. In particular, if $c_i$ is the concentration of the $i$-th species, then $c_i(t,x)$ corresponds to the number of particles of species $i$ within an infinitesimal volume around $x$ (at time $t$). Therefore, the charge contribution within this infinitesimal volume of the $i$-th species is given by $q_ic_i(t,x)$ where $q_i$ is the charge of the $i$-th species. The total charge per infinitesimal volume is then the sum of the contributions of all distinct species, i.e.\
		\begin{gather}
			\nonumber
			\rho_f(t,x)=\sum_{i=1}^Mq_ic_i(t,x),
		\end{gather}
		where we recall that $M$ is the number of distinct species. Similarly, the free current density $J_f$ is the sum of the current densities of all species, $J_f=\sum_{i=1}^MJ_i$, and, in turn, the current density of the $i$-th species corresponds to the expression $J_i=q_iN_i$, where $N_i$ is the particle flux. We therefore obtain the following additional formulas
		\begin{gather}
			\label{S1E11}
			\rho_f=\sum_{i=1}^Mq_ic_i\text{ and }J_f=\sum_{i=1}^Mq_iN_i=-\sum_{i=1}^Mq_iD_i\nabla c_i,
		\end{gather}
		where we used Fick's first law in the last step. This couples (\ref{S1E10}) and (\ref{S1E5}).
		\subsection{Functional assumptions}\label{subssec:functional_assumptions}
		We recall first that according to (\ref{S1E5}), (\ref{S1E10}) and (\ref{S1E11}) we have
		\begin{gather}
			\label{S1E14}
			\partial_tc_i-\operatorname{div}(D_i\nabla c_i)=g_i,\qquad\operatorname{div}(\epsilon \nabla \phi)=\sum_{i=1}^Mq_ic_i=q\cdot c,
		\end{gather}
		where $q=(q_1,\dots,q_M)$, $c=(c_1,\dots,c_M)$ and $\cdot$ denotes the standard Euclidean scalar product.
		
		We observe that the $g_i$ in (\ref{S1E14}) represent the source terms, which are in turn related to the particle production\slash annihilation in the chemical processes involved in the electrolyser. These processes depend on the temperature as well as on the concentrations of the particles. We thus assume that the $g_i$ are of the form $g_i=g_i(t,c,T,x)$. Similarly, the diffusion of the particles is influenced by the particle concentrations and the temperature of the system and thus we suppose that $D_i=D_i(t,c,T,x)$. From these assumptions and (\ref{S1E14}) we see that
		\begin{itemize}[noitemsep]
			\item the potential $\phi$  is coupled to the concentrations $c_i$;
			\item and the concentrations are coupled to $T$ through $g_i$ and $D_i$. 
		\end{itemize}
		As mentioned in \Cref{Inverse}, we assume that there is a phenomenological relation between the potential $\phi$ and the temperature and concentrations, i.e.\  $\phi=\phi(t,c,T,x)$. We make the additional assumption that the electric permeability is not affected much by the change in temperature and particle concentrations within the device and that it remains constant in time, i.e.\ $\epsilon=\epsilon(x)$. Note that in reality $\epsilon$  depends on temperature, and so a more accurate model should include this dependence \cite[Chapter 4.4.1]{Griff13}. Finally, we take the concentrations $c=(c_1,\dots,c_M)$ and the temperature $T$ to be our phenomenologically independent quantities in the sense that we suppose that $c=c(t,x)$ and $T=T(t,x)$. The remaining quantities $g_i$, $D_i$, and $\phi$ may be expressed as functions of $c$ and $T$ (as well as $t$ and $x$).
		
		We observe that from a physical perspective it is reasonable to assume that $0<D_*\leq D_i\leq D^*<\infty$ for some constants $0<D_*<D^*<\infty$ and for all $1\leq i\leq M$. This assumption essentially turns
		\begin{gather}
			\nonumber
			\partial_tc_i-\operatorname{div}(D_i\nabla c_i)=g_i
		\end{gather}
		into a system of parabolic PDEs (or elliptic PDEs in the static case). However, we have an additional temperature dependence so that these continuity equations do not form a complete system of equations.
		
		To ensure that we have a well-behaved system of equations we have to make some additional assumptions on the coefficients appearing in the remaining PDE
		\begin{gather}
			\nonumber
			\operatorname{div}(\epsilon \nabla \phi)=q\cdot c.
		\end{gather}
		Just like in the case of the diffusion coefficients it is reasonable to suppose that
		$\epsilon\in L^{\infty}(\Omega)$ and that $\epsilon\geq \epsilon_*>0$ for a.e. $x\in \Omega$ for some constant $\epsilon_*>0$. In addition, once we make our phenomenological structural assumption $\phi=\phi(t,c(t,x),T(t,x),x)$ we may compute
		\begin{gather}
			\nonumber
			\nabla_x (\phi(t,c(t,x),T(t,x),x))=(\partial_T\phi)(t,c(t,x),T(t,x),x)(\nabla T)(t,x)
			\\
			\nonumber
			+\sum_{i=1}^M(\partial_{c_i}\phi)(t,c(t,x),T(t,x),x)(\nabla c_i)(t,x)+(\nabla_x\phi)(t,c(t,x),T(t,x),x).
		\end{gather}
		Consequently the prefactor of $\nabla T$ in the equation $\operatorname{div}(\epsilon \nabla \phi)=q\cdot c$ is $\epsilon (\partial_T\phi)$. So if we want to obtain a well-behaved equation for $T$ we should require $0<a\leq (\partial_T\phi)\leq b<\infty$ for suitable constants $0<a\leq b<\infty$. This assumption is motivated from a mathematical perspective. A physical interpretation of this assumption can be obtained as follows.
		\newline
		\newline
		Imagine we are given two (identical) electrolyser devices which occupy two bounded regions $\Omega_1$ and $\Omega_2$, where $\Omega_2=d+\Omega_1$ for some suitable $d\in \mathbb{R}^3$, and the distance between them is large enough that their electromagnetic fields do not influence each other. Since the potentials $\phi_1,\phi_2$ on $\Omega_1$, $\Omega_2$ are only unique up to constants we can fix some reference point $x_1\in \partial \Omega_1$, some reference temperature $T^{*}$ and some reference concentration $c^{*}=(c^{*}_1,\dots,c^{*}_M)$ and normalise $\phi_1$ on $\Omega_1$ such that at the initial time $t=0$, $\phi_1(0,c^{*},T^{*},x_1)=0$. Similarly, we set $x_2:=x_1+d$ and use the normalisation $\phi_2(0,c^{*},T^{*},x_2)=0$. We suppose now that we may prescribe the temperature profile $T_0(x)=T(0,x)$ and the concentration profile $c_0(x)=c(0,x)$ along the boundaries of our devices at the initial time $t=0$. We then fix some $T_2>T_1$ and pick our profiles such that $T_0(x_1)=T^{*}=T_0(x_2)$, $c_0(x_1)=c^{*}=c_0(x_2)$ and for some fixed $y_1\in \partial\Omega\setminus \{x_1\}$, letting $y_2:=d+y_1$, we prescribe $T_0(y_1)=T_1$, $T_0(y_2)=T_2$ and ensure that $c_0(y_1)=c_0(y_2)$. Then our condition essentially demands that if we compare the voltage measurement between the points $x_1$ and $y_1$ at $t=0$ with the voltage measurement between $x_2$ and $y_2$ at $t=0$ we will find that the value between $x_2$ and $y_2$ is higher than the value between $x_1$ and $y_1$.  
		\newline
		\newline
		Let us finally mention that it is customary in the chemistry literature to decompose the potential $\phi$ in distinct contributions such as
		\begin{gather}
			\label{S1E15}
			\phi=\phi_{\operatorname{OCV}}+\phi_{\operatorname{act}}+\phi_{\operatorname{conc}}+\phi_{\operatorname{ohm}}
		\end{gather}
		where $\phi_{\operatorname{OCV}}$ is the open circuit voltage, $\phi_{\operatorname{act}}$ is the activation over-potential, $\phi_{\operatorname{conc}}$ is the concentration over-potential and $\phi_{\operatorname{ohm}}$ denotes the ohmic over-potential, cf.\ \cite{Lawand24}. Sometimes additional terms are added, cf.\ \cite{MajHaasEllNaz23}. See also \cite{SeiMitBon25} for a discussion about the interpretation of some of these potentials.
		
		In our approach we deliberately do not perform any such decomposition but instead try to understand the type of measurements that are necessary to be able to reconstruct the dependence of the full potential $\phi$ on quantities such as the temperature and concentrations. We further point out that, in view of the ideal gas law, the pressure of the system should be determined by the concentrations and the temperature of the system and therefore a pressure dependence is implicitly taken into account in our approach.
		\setcounter{equation}{0}\renewcommand\theequation{B\arabic{equation}}
		\section{An extended model including electric forces}
		\label{Electrostatic}
		While modelling the evolution of the ion concentrations, cf.\ \Cref{Concetrations}, we neglected electric forces, which may, however, be of relevance in real life electrolyser devices. Therefore, a more realistic model should include these forces. Following Teorell's work \cite{Teo35}, see also \cite{GorSarWah11} for other types of diffusion models, one can make the following more realistic ansatz
		\begin{gather}
			\label{S1E12}
			N_i=-D_i\nabla c_i+\mu_iq_ic_iE
		\end{gather}
		where $N_i$ denotes the particle flux of the $i$-th species, $D_i$ is the diffusion coefficient of the $i$-th particle, $c_i$ is the concentration of the $i$-th particle (in number of particles per volume), $\mu_i$ is the mobility of the $i$-th species, $q_i$ is the charge of the $i$-th species and $E$ is the electric field.
		
		Inserting this expression into Fick's second law (\ref{S1E4}) gives a modified continuity equation of the form
		\begin{gather}
			\label{S1E13}
			\partial_tc_i+\operatorname{div}(\mu_ic_iq_iE-D_i\nabla c_i)=g_i.
		\end{gather}
		In the present work, we focus on the simpler model (\ref{S1E5}). Future work may explore the more realistic model (\ref{S1E13}).
		\bibliographystyle{plain}
		\bibliography{References}

	\end{document}